\documentclass[reqno]{amsart}

\usepackage{enumerate}
\usepackage{tabto}
\usepackage[mathscr]{euscript}
\usepackage{xcolor}
\usepackage{layout}
\usepackage{fancyhdr}
\usepackage{array}
\usepackage{amsfonts}
\usepackage{amsmath}
\usepackage{amssymb}
\usepackage{mathtools}
\usepackage{graphicx}
\usepackage{bm}
\usepackage{enumitem}
\usepackage{caption} 
\usepackage{color}
\usepackage{kotex}
\usepackage{csquotes}
\usepackage{bookmark}
\usepackage{float}
\usepackage{multirow}
\usepackage[square,numbers,sort&compress]{natbib}
\usepackage{hyperref}
\hypersetup{colorlinks=true,linkcolor=blue,citecolor=red}
\allowdisplaybreaks

\def\Xint#1{\mathchoice
{\XXint\displaystyle\textstyle{#1}}%
{\XXint\textstyle\scriptstyle{#1}}%
{\XXint\scriptstyle\scriptscriptstyle{#1}}%
{\XXint\scriptscriptstyle\scriptscriptstyle{#1}}%
\!\int}
\def\XXint#1#2#3{{\setbox0=\hbox{$#1{#2#3}{\int}$ }
\vcenter{\hbox{$#2#3$ }}\kern-.6\wd0}}

\def\dashint{\Xint-}

\newtheorem{theorem}{Theorem}[section]
\newtheorem{lemma}[theorem]{Lemma}

\newtheorem{remark}[theorem]{Remark}
\theoremstyle{definition}
\newtheorem{definition}[theorem]{Definition}

\numberwithin{equation}{section}

\newcommand{ \mr }{ \mathbb{R} }

\newcommand{\iints}[1]{{\int\hspace{-0.28cm}\int_{#1}}}
\newcommand{\iintss}{{\int\hspace{-0.28cm}\int}}
\newcommand{ \miints }{{\iintss\hspace{-0.56cm} -\hspace{-0.15cm}-}}
\newcommand{\miint}[1]{{\miints_{\hspace{-0.13cm}#1}}}

\begin{document}
\title[Interpolative refinement of gap bounds]{Interpolative Refinement of Gap Bound Conditions for Singular Parabolic Double Phase Problems}

\author{Bogi Kim}\address{Department of Mathematics, Kyungpook National University, Daegu, 41566, Republic of Korea} \email{rlaqhrl4@knu.ac.kr} \author{Jehan Oh}\address{Department of Mathematics, Kyungpook National University, Daegu, 41566, Republic of Korea} \email{jehan.oh@knu.ac.kr}

\date{\today.}
\thanks{This work is supported by National Research Foundation of Korea (NRF) grant funded by the Korea government [Grant Nos. RS-2023-00217116, RS-2025-00555316, RS-2025-25415411, and RS-2025-25426375].}

\begin{abstract}
We consider inhomogeneous singular parabolic double phase equations of type
$$
u_t-\operatorname{div}(|Du|^{p-2}Du + a(x,t)|Du|^{q-2}Du)=-\operatorname{div} (|F|^{p-2}F + a(x,t)|F|^{q-2}F)
$$
in $\Omega_T\coloneq \Omega \times (0,T)\subset \mathbb{R}^n\times \mathbb{R}$, where $\frac{2n}{n+2}<p\leq 2$, $p<q$ and $0\leq a(\cdot)\in C^{\alpha,\frac{\alpha}{2}}(\Omega_T)$.
We establish gradient higher integrability results for weak solutions to the above problems under one of the following two assumptions:
$$
u\in L^\infty (\Omega_T) \quad\text{and}\quad q\leq p +\frac{\alpha(p(n+2)-2n)}{4},
$$
or
$$
u\in C(0,T;L^s(\Omega)),\quad s\geq 2 \quad\text{and}\quad q\leq p+\frac{\alpha \mu_s}{n+s},
$$
where $\mu_s\coloneq \frac{(p(n+2)-2n)s}{4}$. These results yield an interpolation refinement of gap bounds in the singular parabolic double phase setting.
\end{abstract}
\maketitle

\section{\bf Introduction}\label{section 1}
In this paper, we investigate the local gradient higher integrability of weak solutions to inhomogeneous singular parabolic double phase problems with the model equation
\begin{equation}\label{main model equation}
    \begin{aligned}
        &u_t-\operatorname{div} (|Du|^{p-2}Du + a(x,t)|Du|^{q-2}Du)\\
        &\qquad\qquad\qquad =-\operatorname{div} (|F|^{p-2}F + a(x,t)|F|^{q-2}F) \qquad \text{in } \Omega_T\coloneq \Omega\times (0,T),
    \end{aligned}
\end{equation}
where $n\geq 2$, $\frac{2n}{n+2}<p \leq 2$, $p<q$, $T>0$, $\Omega$ is a bounded open set in $\mr^n$ and $a(\cdot)\in C^{\alpha,\frac{\alpha}{2}}(\Omega_T)$ is non-negative. Here, $a(\cdot)\in C^{\alpha,\frac{\alpha}{2}}(\Omega_T)$ means that $a(\cdot)\in L^\infty (\Omega_T)$ and there exists a H\"{o}lder constant $[a]_\alpha\coloneq [a]_{\alpha,\frac{\alpha}{2};\Omega_T}>0$ such that
$$
|a(x_1,t_1)-a(x_2,t_2)|\leq [a]_\alpha \max\left\{|x_1-x_2|^\alpha,|t_1-t_2|^\frac{\alpha}{2}\right\}
$$
for all $x_1,\,x_2\in\Omega$ and $t_1,\,t_2\in(0,T)$. The elliptic version of \eqref{main model equation} is as follows:
$$
\begin{aligned}
    -\operatorname{div} (|Du|^{p-2}Du + a(x)|Du|^{q-2}Du)=-\operatorname{div}(|F|^{p-2}F + a(x)|F|^{q-2}F) \qquad \text{in } \Omega.
\end{aligned}
$$
Here, $1<p\leq q$ and $0\leq a(\cdot)\in C^{\alpha}(\Omega)$ for some $\alpha \in (0,1]$. This equation is the Euler-Lagrange equation of
$$
W^{1,1}(\Omega)\ni w \mapsto \int_\Omega \left[\frac{1}{p}|Dw|^p +\frac{1}{q}a(x)|Dw|^q-\left\langle |F|^{p-2}F+a(x)|F|^{q-2}F, Dw\right\rangle\right]\, dx,
$$
and is called the elliptic double phase problem. It was first introduced in \cite{Zhikov1986,Zhikov1993,Zhikov1995,Zhikov1997} as an example exhibiting the Lavrentiev phenomenon and as a model explaining homogenization in strongly anisotropic materials. Furthermore, various variants of the double phase problem are used in a wide range of applied science fields, including transonic flows \cite{Bahrouni2019}, quantum physics \cite{Benci2000}, steady-state reaction–diffusion systems \cite{cherfils2005stationary}, image denoising and processing \cite{Charkaoui2024,Chen2006,Harjulehto2013,Harjulehto2021,Kbiri2014,Fang2010}, and heat diffusion in materials with heterogeneous thermal properties \cite{Arora2023}, and so on. To study regularity properties of the elliptic double phase problem, we need a condition relating the closeness of $p$ and $q$ to the H\"{o}lder exponent $\alpha$ of the modulating coefficient $a(\cdot)$, see for instance \cite{Mingione2021}. Indeed, according to \cite{Baroni2018,Colombo2015,Esposito2004,Colombo2015a}, when the gap bound condition either
\begin{equation}\label{cond : gap bound condition of Elliptic double phase problem}
\frac{q}{p}\leq 1 + \frac{\alpha}{n}
\end{equation}
or
\begin{equation}\label{cond : gap bound condition of bounded solution of Elliptic double phase problem}
u\in L^\infty (\Omega) \quad \text{and}\quad q\leq p+\alpha
\end{equation}
is satisfied, a weak solution $u$ and its gradient $Du$ are H\"{o}lder continuous. Also, Baroni-Colombo-Mingione \cite{Baroni2018} established that the gradient of $u$ is H\"{o}lder continuous under the assumption
$$
u\in C^{\gamma}(\Omega)\quad \text{and}\quad q\leq p +\frac{\alpha}{1-\gamma}\quad\text{with }\gamma\in (0,1).
$$
Furthermore, Ok \cite{Ok2020} proved that if
$$
u\in L^\gamma_{\operatorname{loc}}(\Omega) \quad \text{and}\quad q\leq p +\frac{\gamma \alpha}{n+\gamma}
$$
for $p\in (1,n)$ and $\gamma>\frac{np}{n-p}$, then a local quasi-minimizer $u$ of elliptic double phase problems is locally H\"{o}lder continuous. From these results, one may expect that imposing stronger regularity assumptions on $u$ allows one to relax the gap bound condition while still obtaining the same regularity results for $u$. In particular, the results in \cite{Ok2020} lead to an interpolation of the gap bound conditions. On the other hand, under the condition either \eqref{cond : gap bound condition of Elliptic double phase problem} or \eqref{cond : gap bound condition of bounded solution of Elliptic double phase problem}, a variety of regularity results have been studied. For instance, Baroni-Colombo-Mingione \cite{Baroni2015} and Ok \cite{Ok2017} established Harnack's inequality and H\"{o}lder continuity for weak solutions. Also, Baasandorj-Byun-Oh \cite{Baasandorj2020}, Colombo-Mingione \cite{Colombo2016} and De Filippis-Mingione \cite{DeFilippis2019} obtained Calder\'{o}n-Zygmund type estimates. In addition, various regularity results for elliptic double phase problems can be found in \cite{Byun2021,Byun2021a,Byun2017,Byun2020,Haestoe2022,Haestoe2022a,Kim2024a,Kim2026,Kim2025a,DeFilippis2019,2023DeFilippis}, and so on.

To discuss the regularity of weak solutions to the parabolic double phase problems, a gap bound condition is also required. In fact, for the degenerate parabolic double phase problems, i.e., when $p\geq 2$, Kim-Kinnunen-Moring \cite{2023_Gradient_Higher_Integrability_for_Degenerate_Parabolic_Double-Phase_Systems} established that the (spatial) gradient of the solution satisfies higher integrability results under the gap bound condition
\begin{equation}\label{cond : the degenerate case assumption}
q\leq p +\frac{2\alpha}{n+2}.
\end{equation}
Also, under \eqref{cond : the degenerate case assumption}, Kim-Kinnunen-S\"{a}rki\"{o} \cite{Wontae2023a} studied energy estimates and the existence theory for weak solutions, see also \cite{Chlebicks2019,Singer2016}. On the other hand, for the singular parabolic double phase problems, i.e., when $\frac{2n}{n+2}<p\leq 2$, Kim \cite{Wontae2024} and Kim-S\"{a}rki\"{o} \cite{Wontae2023b} obtained gradient higher integrability results and Calder\'{o}n-Zygmund type estimates under the gap bound condition
\begin{equation}\label{cond : the sigular case assumption}
q\leq p +\frac{\mu_2\alpha}{n+2},
\end{equation}
where $\mu_2 = \frac{p(n+2)-2n}{2}$. Also, H\"{a}sto-Ok \cite{2025Hasto} established gradient higher integrability results not only for both degenerate and singular cases, but also for problems with generalized Orlicz growth. In addition, regularity results on parabolic double phase problems can be found in \cite{Buryachenko2022,Kim2024,Kim2025,Sen2025}.

Now, we introduce the main equations and theorems. The main equations under consideration are of the form 
\begin{equation}\label{eq : the main equation}
    u_t-\operatorname{div}\mathcal{A}(z,Du)=-\operatorname{div}\mathcal{B}(z,F)\qquad \text{in }\Omega_T.
\end{equation}
Here, $\mathcal{A}:\Omega_T \times \mr^{n}\rightarrow \mr^{n}$ is a Carath\'{e}odory vector field satisfying that there exist constants $0<\nu\leq L <\infty$ such that
\begin{equation}    \label{cond : double phase bounded condition of integrand}
    \mathcal{A}(z,\xi)\cdot \xi\geq \nu (|\xi|^p+a(z)|\xi|^q)\quad \text{and}\quad |\mathcal{A}(z,\xi)|\leq L(|\xi|^{p-1}+a(z)|\xi|^{q-1})
\end{equation}
for all $z\in\Omega_T$ and $\xi\in \mr^n$. We also assume that $\mathcal{B}:\Omega_T\times\mr^n\rightarrow\mr^n$ is a Carath\'{e}odory vector field satisfying
\begin{equation}\label{cond : growth condition of source term}
    |\mathcal{B}(z,\xi)|\leq L(|\xi|^{p-1}+a(z)|\xi|^{q-1})
\end{equation}
for all $z\in\Omega_T$ and $\xi\in \mr^n$. For simplicity, we denote $H(z,\varkappa)\coloneq  \varkappa^p +a(z)\varkappa^q$ for $\varkappa\geq 0$ and $z\in\Omega_T$. The definition of a weak solution to \eqref{eq : the main equation} is as follows:
\begin{definition}
    A function $u:\Omega_T \rightarrow \mr$ with
    $$
    u\in C(0,T;L^2(\Omega))\cap L^{1}(0,T;W^{1,1}(\Omega))
    $$
    and
    $$
    \iints{\Omega_T} H(z,|Du|)\, dz <\infty
    $$
    is a weak solution to \eqref{eq : the main equation}, if 
    $$
    \iints{\Omega_T} (-u\cdot\varphi_t +\mathcal{A}(z,Du)\cdot D\varphi)\, dz=\iints{\Omega_T} \mathcal{B}(z,F)\cdot D\varphi\, dz
    $$
    holds for every $\varphi\in C_0^\infty (\Omega_T)$.
\end{definition}

We aim to prove gradient higher integrability results for bounded solutions to singular parabolic double phase problems, and to establish an interpolation result between \eqref{cond : the sigular case assumption} and the assumption on bounded solutions. Indeed, for the degenerate case, Kim-Oh \cite{kim2025boundedsolutionsinterpolativegap} established the interpolation of gap bound conditions. In this paper, we first assume that
\begin{equation}\label{cond : main assumption with infty}
    u\in L^\infty(\Omega_T)\quad\text{and}\quad q\leq p+\frac{\alpha (p(n+2)-2n)}{4}
\end{equation}
for some $\alpha \in (0,1]$. Unlike the degenerate case, this depends on the dimension. In fact, $\frac{p(n+2)-2n}{4}$ is the standard scaling deficit arising in singular parabolic $p$-Laplace problems, see for instance \cite[Section VIII]{1993_Degenerate_parabolic_equations_DiBenedetto}. According to \cite{kim2025boundedsolutionsinterpolativegap} and \cite{Wontae2025}, for degenerate parabolic double phase problems, under the condition
\begin{equation}\label{cond : degenerate bounded solution condition}
u\in L^\infty(\Omega_T)\quad\text{and}\quad q\leq p+\alpha,
\end{equation}
it was shown that weak solutions satisfy gradient higher integrability results and are locally H\"{o}lder continuous. We note that if $p=2$ in \eqref{cond : main assumption with infty}, then 
$$
q\leq p+\frac{\alpha (p(n+2)-2n)}{4} = p+\alpha.
$$
Therefore, \eqref{cond : main assumption with infty} and \eqref{cond : degenerate bounded solution condition} are connected in the sense that they coincide when $p=2$. When \eqref{cond : main assumption with infty} holds, we assume that the source term $F:\Omega_T\rightarrow\mr^n$ satisfies 
\begin{equation}\label{cond : source term with infty}
    H(\cdot,|F|) \in L^{\gamma_b}(\Omega_T), \quad \text{where } \gamma_b\coloneq \frac{n+2}{2}. 
\end{equation}
Now, to introduce our first main theorem, we write a collection of parameters as
$$
\begin{aligned}
    \operatorname{data}_b := &(n,p,q,\alpha,\nu,L,[a]_\alpha,\operatorname{diam}(\Omega), |\Omega_T|, \|u\|_{L^\infty(\Omega_T)}, \\
    &\quad \|H(z,|Du|)\|_{L^1(\Omega_T)}, \|H(z,|F|)\|_{L^{\gamma_b}(\Omega_T)}).
\end{aligned}
$$
\begin{theorem}\label{thm : main theorem for infty}
    Assume that \eqref{cond : main assumption with infty} and \eqref{cond : source term with infty} are satisfied, and let $u$ be a weak solution to \eqref{eq : the main equation}. Then there exist constants $\varepsilon_0=\varepsilon_0(\operatorname{data}_b)>0$ and $c=c(\operatorname{data}_b,$ $\|a\|_{L^\infty(\Omega_T)})>1$ such that
    $$
    \begin{aligned}
        \miint{Q_r(z_0)} H(z,|Du|)^{1+\varepsilon}\, dz&\leq c\left(\miint{Q_{2r}(z_0)} H(z,|Du|)\, dz\right)^{1+\frac{2q\varepsilon}{p(n+2)-2n}}\\
        &\qquad +c\left(\miint{Q_{2r}(z_0)} [H(z,|F|)+1]^{1+\varepsilon}\, dz\right)^{\frac{2q}{p(n+2)-2n}}
    \end{aligned}
    $$
    for every $Q_{2r}(z_0)\subset \Omega_T$ and $\varepsilon\in (0,\varepsilon_0)$.
\end{theorem}

Next, to establish an interpolation between \eqref{cond : the sigular case assumption} and \eqref{cond : main assumption with infty}, we assume that
\begin{equation}\label{cond : main assumption with s}
    u\in C(0,T; L^s(\Omega))\quad\text{and}\quad q\leq p+\frac{\alpha \mu_s}{n+s},\quad \text{where } \mu_s=\frac{(p(n+2)-2n)s}{4}
\end{equation}
for some $s\in[2,\infty)$ and $\alpha \in (0,1]$. 
We remark that our gap bound $q-p$ varies continuously from the baseline bound (at $s=2$) to the bounded-solution bound (as $s \to \infty$), yielding a genuine interpolation family.
For the degenerate case, Kim-Oh \cite{kim2025boundedsolutionsinterpolativegap} proved that under the assumption
\begin{equation}\label{cond : degenerate interpolation condition}
    u\in C(0,T;L^s(\Omega))\quad \text{and}\quad q\leq p+\frac{s\alpha}{n+s},
\end{equation}
weak solutions satisfy gradient higher integrability results, and used these results to describe an interpolation.
For the singular case, we note from $\frac{2n}{n+2}<p\leq 2$ that $\mu_s \leq s$, and so $q\leq p+1\leq 3$. Also, $\mu_s\searrow 0$ as $p\searrow \frac{2n}{n+2}$, and $\mu_s=s$ when $p=2$. Moreover, when $p=2$, \eqref{cond : main assumption with s} and \eqref{cond : degenerate interpolation condition} coincide. Lastly, we have $\frac{\alpha \mu_s}{n+s}=\frac{\alpha \mu_2}{n+2}$ for $s=2$, and $\frac{\alpha\mu_s}{n+s}\nearrow \frac{\alpha(p(n+2)-2n)}{4}$ as $s\rightarrow \infty$. Hence, the condition \eqref{cond : main assumption with s} serves as an interpolative condition that links \eqref{cond : the degenerate case assumption}, \eqref{cond : the sigular case assumption}, \eqref{cond : main assumption with infty}, \eqref{cond : degenerate bounded solution condition} and \eqref{cond : degenerate interpolation condition}. Furthermore, this justifies that the bounds in each of these conditions are natural. On the other hand, we impose a stronger assumption on $u$, in contrast to the standard requirement $u \in C(0,T;L^2(\Omega))$, the latter being the function space naturally arising from the presence of the time-derivative term $u_t$ in the definition of a weak solution. Unlike \cite{Ok2020}, the assumption here pertains to the time variable rather than the spatial one, and this is exactly what distinguishes the parabolic case from the elliptic case. Furthermore, under the assumption \eqref{cond : main assumption with s}, we assume that the source term $F:\Omega_T\rightarrow\mr^n$ satisfies
\begin{equation}\label{cond : source term with s}
    H(\cdot,|F|)\in L^{\gamma_s}(\Omega_T),\quad \text{where } \gamma_s\coloneq \frac{s(n+2)}{2(n+s)}.
\end{equation} 
We note that when $s=2$, we have $\gamma_2=1$, and hence $H(\cdot,|F|)$ is in $L^1(\Omega_T)$ as the assumption in \cite{Wontae2024}. We also remark that $\gamma_s\nearrow \gamma_b$ as $s\rightarrow \infty$. 

Now, to state our second main theorem, we write a collection of parameters as
$$
\begin{aligned}
    \operatorname{data}_s := &(n,p,q,s,\alpha,\nu,L,[a]_\alpha,\operatorname{diam}(\Omega), |\Omega_T|, \|u\|_{C(0,T;L^s(\Omega))}, \\
    &\quad \|H(z,|Du|)\|_{L^1(\Omega_T)}, \|H(z,|F|)\|_{L^{\gamma_s}(\Omega_T)}).
\end{aligned}
$$
\begin{theorem}\label{thm : main theorem for s<infty}
    Assume that \eqref{cond : main assumption with s} and \eqref{cond : source term with s} are satisfied, and let $u$ be a weak solution to \eqref{eq : the main equation}. Then there exist constants $\varepsilon_0=\varepsilon_0(\operatorname{data}_s)>0$ and $c=c(\operatorname{data}_s,$ $\|a\|_{L^\infty(\Omega_T)})>1$ such that
    $$
    \begin{aligned}
        \miint{Q_r(z_0)} H(z,|Du|)^{1+\varepsilon}\, dz&\leq c\left(\miint{Q_{2r}(z_0)} H(z,|Du|)\, dz\right)^{1+\frac{2q\varepsilon}{p(n+2)-2n}}\\
        &\qquad +c\left(\miint{Q_{2r}(z_0)} [H(z,|F|)+1]^{1+\varepsilon}\, dz\right)^{\frac{2q}{p(n+2)-2n}}
    \end{aligned}
    $$
    for every $Q_{2r}(z_0)\subset \Omega_T$ and $\varepsilon\in (0,\varepsilon_0)$.
\end{theorem}

\begin{remark}
    If we consider the above estimate for every $Q_{2r}(z_0)\subset \Omega_T$ with $0<r\leq 1$, then $\operatorname{diam}(\Omega)$ is not required among the parameters in $\operatorname{data}_b$ and $\operatorname{data}_s$. Moreover, if $\|H(z,|F|)\|_{L^1(\Omega_T)}$ is included in $\operatorname{data}_b$ and $\operatorname{data}_s$, then $|\Omega_T|$ in $\operatorname{data}_b$ and $\operatorname{data}_s$ can be removed, see the proof of Lemma \ref{lem : decay estimate}.
\end{remark}
\begin{remark}
    The gradient higher integrability results obtained in this work can be extended to parabolic double phase systems under analogous structural assumptions. However, in order to keep the presentation concise, we confine our analysis to the scalar equation case.
\end{remark}

\begin{remark}
    Kim-Oh \cite{kim2025boundedsolutionsinterpolativegap} considered the homogeneous degenerate parabolic double phase problems with the model equation
    $$
    u_t-\operatorname{div}(|Du|^{p-2}Du+a(x,t)|Du|^{q-2}Du)=0\quad \text{in }\Omega_T,
    $$
    where $2\leq p<q$. As in this paper, one can include a source term by imposing appropriate assumptions. When \eqref{cond : degenerate bounded solution condition} is satisfied, we need the assumption that the source term $F:\Omega_T\rightarrow \mr^n$ satisfies
    $$
    H(\cdot,|F|)\in L^{\tilde{\gamma}_b}(\Omega_T),\quad\text{where } \tilde{\gamma}_b=\frac{n+p}{p},
    $$
    see also \cite{Chlebicka2025}, whereas, when \eqref{cond : degenerate interpolation condition} holds, we need
    $$
    H(\cdot,|F|)\in L^{\tilde{\gamma}_s}(\Omega_T),\quad\text{where } \tilde{\gamma}_s=\frac{(n+p)s+n(p-2)}{p(n+s)}.
    $$
    It is easy to see that these conditions are connected with each other and also with \eqref{cond : source term with infty} and \eqref{cond : source term with s}. One can obtain gradient higher integrability results by following the same arguments as in this paper.
\end{remark}
\begin{remark}
    A noteworthy point is that, in the singular case, the gap bound conditions depend on $p$, whereas the assumptions on the source term do not. In contrast, in the degenerate case, the conditions imposed on the source term depend on $p$, but the gap bound conditions do not.
\end{remark}

Differing from \cite{Wontae2024}, we distinguish the $p$-intrinsic and $(p,q)$-intrinsic cases by  imposing
\begin{equation}\label{p phase and p,q phase}
K\lambda^p\geq \sup_{Q_{10\rho}(z)} a(\cdot) \lambda^q \quad \text{and}\quad K\lambda^p\leq \sup_{Q_{10\rho}(z)} a(\cdot)\lambda^q,
\end{equation}
respectively, where $K>1$ and $\rho$ denotes the radius in the $p$-intrinsic cylinder arising in the stopping-time argument in Section \ref{section 3}. These conditions simplify the proof of the lemmas in Section \ref{section 4}. Furthermore, when \eqref{p phase and p,q phase}$_2$ is satisfied, to obtain the  comparison condition for $a(\cdot)$, we need
$$
\sup_{Q_{10\rho}(z)} a(\cdot) \gtrsim \rho^\alpha.
$$
Hence, we want to show that
$$
K\lambda^p\leq \sup_{Q_{10\rho}(z)} a(\cdot)\lambda^q \quad \text{and}\quad \sup_{Q_{10\rho}(z)} a(\cdot) \lesssim  \rho^\alpha
$$
cannot hold simultaneously. Under the assumptions \eqref{cond : main assumption with infty} and \eqref{cond : source term with infty}, or \eqref{cond : main assumption with s} and \eqref{cond : source term with s}, the argument used in \cite{Wontae2024} can no longer be employed to prove this. To address this issue, we use Lemma \ref{lem : decay estimate} (see also \cite[Lemma 3.1]{Wontae2024}) to prove Lemmas \ref{lem : no occurence with s=infty} and \ref{lem : no occurence with s<infty}. In particular, the $F$-term is controlled by using \eqref{cond : source term with infty} or \eqref{cond : source term with s}. These conditions are used only in Lemmas \ref{lem : no occurence with s=infty} and \ref{lem : no occurence with s<infty}. In Section \ref{section 3}, we employ a stopping time argument to derive the properties of $p$- and $(p,q)$-intrinsic cylinders defined in Section \ref{section 2}. In Section \ref{section 4}, we prove the reverse H\"{o}lder inequalities for each intrinsic cylinder. In particular, for the $p$-intrinsic cylinder, we first establish the case $s=\infty$. For the case $s<\infty$, in order to prove Lemma \ref{lem : estimate of the first term in Lemma 2.1 in p-intrinsic cylinder with s}, we divide the argument into the two subcases $2\leq s \leq 4$ and $4<s<\infty$. Lastly, using the Vitali covering lemma (see Subsection \ref{subsection 5.1}) and Fubini's theorem, we prove Theorems \ref{thm : main theorem for infty} and \ref{thm : main theorem for s<infty} in Subsection \ref{subsection 5.2}.

\section{\bf Preliminaries}\label{section 2}
For a fixed point $z_0 \in \Omega_T$, we denote
\begin{equation}\label{def : definition of H with a fixed center z_0}
    H_{z_0}(\varkappa)\coloneq \varkappa^p+a(z_0)\varkappa^q \qquad \text{for } \varkappa\geq 0.
\end{equation}
We write parabolic cylinders as
$$
Q_{R,\ell}(z_0)=B_R(x_0)\times (t_0-\ell,t_0+\ell),\quad R,\,\ell>0,
$$
and
$$
Q_\rho (z_0)=B_\rho (x_0)\times I_\rho (t_0),
$$
where
$$
B_\rho (x_0)=\{x\in \mr^n : |x-x_0|<\rho\}
$$
and
$$
I_\rho(t_0)=(t_0-\rho^2,t_0+\rho^2).
$$
We set a $p$-intrinsic cylinder
\begin{equation}\label{eq : definition of p-intrinsic cylinder}
Q_\rho^\lambda(z_0)\coloneq  B_\rho^\lambda(x_0)\times I_{\rho}(t_0),\quad\text{where } B_\rho^\lambda (x_0)= B_{\lambda^{\frac{p-2}{2}}\rho}(x_0)
\end{equation}
and a $(p,q)$-intrinsic cylinder
\begin{equation}\label{eq : definition of p,q-intrinsic cylinder}
G_\rho^\lambda(z_0)\coloneq  B_\rho^\lambda(x_0)\times J_\rho^\lambda(t_0),\quad\text{where } J_\rho^\lambda(t_0)=\left(t_0-\frac{\lambda^p}{H_{z_0}(\lambda)}\rho^2,t_0+\frac{\lambda^p}{H_{z_0}(\lambda)}\rho^2\right).
\end{equation}
Since $\frac{\lambda^p}{H_{z_0}(\lambda)}\rho^2=\frac{\lambda^2}{H_{z_0}(\lambda)}(\lambda^{\frac{p-2}{2}}\rho)^2$, we see that $G_\rho^\lambda(z_0)$ is the standard intrinsic cylinder for a $(p,q)$-Laplace problem. For $c>0$, we denote
$$
cQ_\rho^\lambda (z_0)=Q_{c\rho}^\lambda (z_0)\quad \text{and}\quad cG_\rho^\lambda(z_0)=G_{c\rho}^\lambda (z_0).
$$
The integral average of $f\in L^1(\Omega_T)$ over a measurable set $E\subset \Omega_T$ with $0<|E|<\infty$ is denoted by
$$
f_E=\frac{1}{|E|}\iints{E} f\, dz=\miint{E} f\, dz.
$$
Also, the spatial integral average of $f\in C(0,T;L^1(\Omega))$ over an $n$-dimensional ball $B\subset \Omega$ is denoted by
$$
f_{B}(t)=\dashint_{B} f(x,t) \, dx.
$$
For convenience, we write 
\begin{equation*}
    \operatorname{data}=\begin{cases}
        \operatorname{data}_b &\text{if \eqref{cond : main assumption with infty} holds,}\\
        \operatorname{data}_s &\text{if \eqref{cond : main assumption with s} holds}.
    \end{cases}
\end{equation*}
Next, we denote the super-level sets as 
\begin{equation}\label{def : definition of Psi}
    \Psi(\Lambda)\coloneq \{z\in\Omega_T:H(z,|Du(z)|)>\Lambda\}
\end{equation}
and
\begin{equation}\label{def : definition of Phi}
    \Phi(\Lambda)\coloneq \{z\in\Omega_T:H(z,|F(z)|)>\Lambda\}.
\end{equation}

The following two lemmas are derived from the definition of weak solution to \eqref{eq : the main equation}. However, a priori condition $u\in L^1(0,T;W^{1,1}(\Omega))$ with
$$
\iints{\Omega_T} H(z,|Du|)\, dz <\infty
$$
does not allow $u$ to be used as a test function in the definition of a weak solution. However, through a Lipschitz truncation method, $u$ can be used as a test function, as in the degenerate case \cite{Wontae2023a}. The proof of the following lemmas can be found in \cite{Wontae2023a} and \cite{2023_Gradient_Higher_Integrability_for_Degenerate_Parabolic_Double-Phase_Systems}. Here, the estimate for the source term is obtained by first using \eqref{cond : growth condition of source term} and then proceeding with the proof in the same manner.
\begin{lemma}[\cite{Wontae2024}, Lemma 2.3]\label{lem : Caccioppoli inequality}
    Let $u$ be a weak solution to \eqref{eq : the main equation}. Then there exists a positive constant $c=c(n,p,q,\nu,L)$ such that
    $$
    \begin{aligned}
        &\sup_{t\in (t_0-\tau,t_0+\tau)}\dashint_{B_r(x_0)} \frac{|u-u_{Q_{r,\tau}(z_0)}|^2}{\tau}\, dx+\miint{Q_{r,\tau}(z_0)} H(z,|Du|)\, dz\\
        &\quad \leq c \miint{Q_{R, \ell}\left(z_0\right)}\left(\frac{\left|u-u_{Q_{R, \ell}\left(z_0\right)}\right|^p}{(R-r)^p}+a(z) \frac{\left|u-u_{Q_{R, \ell}\left(z_0\right)}\right|^q}{(R-r)^q}\right) dz \\
        &\qquad+c \miint{Q_{R, \ell}\left(z_0\right)} \frac{\left|u-u_{Q_{R, \ell}\left(z_0\right)}\right|^2}{\ell-\tau} \, dz + c\miint{Q_{R,\ell}(z_0)} H(z,|F|)\, dz
    \end{aligned}
    $$
    for every $Q_{R, \ell}\left(z_0\right) \subset \Omega_T$ with $R, \ell>0,\, r \in[R / 2, R)$ and $\tau \in [\ell / 2^2, \ell )$.
\end{lemma}
\begin{lemma}[\cite{Wontae2024}, Lemma 2.4] \label{lem : semi-Parabolic Poincare inequality}
    Let $u$ be a weak solution to \eqref{eq : the main equation}. Then there exists a positive constant $c=c(n, m, L)$ such that
    $$
    \begin{aligned}
        &\miint{Q_{R, \ell}(z_0)} \frac{|u-u_{\mathcal{Q}_{R, \ell}(z_0)}|^{\theta m}}{R^{\theta m}} \,dz \leq c \miint{Q_{R, \ell}(z_0)}|D u|^{\theta m} \,dz \\
        &\qquad+c\left(\frac{\ell}{R^2} \iints{Q_{R, \ell}\left(z_0\right)} \left[|D u|^{p-1}+a(z)|D u|^{q-1}+|F|^{p-1}+a(z)|F|^{q-1}\right] d z\right)^{\theta m}
    \end{aligned}
    $$
    for every $Q_{R, \ell}(z_0) \subset \Omega_T$ with $R, \ell>0, m \in(1, q]$ and $\theta \in(1 / m, 1]$.
\end{lemma}
\section{\bf Stopping time argument}\label{section 3}
We put
\begin{equation}\label{def : lambda_0 and Lambda_0}
\begin{aligned}
    \lambda_0^\frac{p(n+2)-2n}{2}&\coloneq \miint{Q_{2r}(z_0)} \left[H(z,|Du|)+H(z,|F|)+1\right] dz,\\ 
    &\Lambda_0\coloneq \lambda_0^p+\sup_{z\in Q_{2r}(z_0)} a(z)\lambda_0^q,
\end{aligned}
\end{equation}
where $Q_{2r}(z_0)=B_{2r}(x_0)\times (t_0-(2r)^2,t_0+(2r)^2)$. Moreover, let
\begin{equation}\label{def : definition of K and kappa}
    K\coloneq \begin{cases}
        1+80c_b[a]_\alpha &\text{if \eqref{cond : main assumption with infty} holds,}\\
        1+80c_s[a]_\alpha &\text{if \eqref{cond : main assumption with s} holds,}
    \end{cases}\quad \text{and}\quad \kappa\coloneq 20K,
\end{equation}
where $c_b$ and $c_s$ will be defined in Lemmas \ref{lem : no occurence with s=infty} and \ref{lem : no occurence with s<infty}, respectively.
For $\Psi(\Lambda)$ as in \eqref{def : definition of Psi}, $\Phi(\Lambda)$ as in \eqref{def : definition of Phi} and $\varrho\in [r,2r]$, we write
$$
\Psi(\Lambda,\varrho)\coloneq  \Psi(\Lambda)\cap Q_\varrho(z_0)=\{z\in Q_\varrho(z_0):H(z,|Du(z)|)>\Lambda\}
$$
and
$$
\Phi(\Lambda,\varrho)\coloneq  \Phi(\Lambda)\cap Q_\varrho(z_0)=\{z\in Q_\varrho(z_0):H(z,|F(z)|)>\Lambda\}.
$$

Next, we apply a stopping time argument. Let $r\leq r_1<r_2\leq 2r$ and
$$
\Lambda>\left(\frac{4\kappa r}{r_2-r_1}\right)^{\frac{2q(n+2)}{p(n+2)-2n}}\Lambda_0,
$$
where $\kappa$ is defined in \eqref{def : definition of K and kappa}. For any $w\in \Psi(\Lambda,r_1)$, we choose $\lambda_w>0$ such that
\begin{equation}\label{cond : Lambda=H(lambda)}
\Lambda=\lambda_w^p+a(w)\lambda_w^q=H_w(\lambda_w),
\end{equation}
where $H_{w}$ denotes the function defined in \eqref{def : definition of H with a fixed center z_0} with $z_0$ replaced by $w$. According to \cite[Subsection 4.1]{Wontae2024}, we obtain that there exists $\varrho_w \in (0,(r_2-r_1)/2\kappa)$ such that
\begin{equation}\label{cond : integral of H in =varrho}
\miint{Q_{\varrho_w}^{\lambda_w}(w)} [H(z,|Du|)+H(z,|F|)]\, dz =\lambda_w^p
\end{equation}
and
\begin{equation}\label{cond : integral of H in >varrho}
\miint{Q_{\varrho}^{\lambda_w}(w)} [H(z,|Du|)+H(z,|F|)]\, dz <\lambda_w^p
\end{equation}
for any $\varrho\in(\varrho_w,r_2-r_1)$. 

For $K>1$ as in \eqref{def : definition of K and kappa}, we consider the following three cases: 
\begin{enumerate}
    \item\label{case : p-phase} $\displaystyle K\lambda_w^p\geq \sup_{Q_{10\varrho_w}(w)}a(\cdot)\lambda_w^q$,
    \item\label{case : p,q-phase} $\displaystyle K\lambda_w^p\leq \sup_{Q_{10\varrho_w}(w)}a(\cdot)\lambda_w^q\quad$ and $\quad\displaystyle \sup_{Q_{10\varrho_w}(w)}a(\cdot)\geq 4[a]_\alpha (10\varrho_w)^\alpha$,
    \item\label{case : no occurence} $\displaystyle K\lambda_w^p\leq \sup_{Q_{10\varrho_w}(w)}a(\cdot)\lambda_w^q\quad$ and $\quad\displaystyle \sup_{Q_{10\varrho_w}(w)}a(\cdot)\leq 4[a]_\alpha (10\varrho_w)^\alpha$.
\end{enumerate}
\textbf{Case \eqref{case : p-phase}}: By using \eqref{cond : integral of H in =varrho} and \eqref{cond : integral of H in >varrho} and replacing the center point $w$, radius $\varrho_w$ and $\lambda_w$ with $z_0$, $\rho$ and $\lambda$, respectively, we obtain
\begin{equation}\label{cond : p-phase condition}
    \left\{\begin{aligned}
        &K\lambda^p\geq \sup_{Q_{10\rho}(z_0)} a(\cdot)\lambda^q,\\
        &\miint{Q_\sigma^\lambda (z_0)} [H(z,|Du|)+H(z,|F|)]\, dz <\lambda^p \quad \text{for any }\sigma\in(\rho,2\kappa\rho],\\
        &\miint{Q_\rho^\lambda(z_0)} [H(z,|Du|)+H(z,|F|)]\, dz = \lambda^p.
    \end{aligned}\right.
\end{equation}

The following lemma provides an estimate for the relationship between $\rho$ and $\lambda$, which will be used later.

\begin{lemma}\label{lem : decay estimate}
    If \eqref{cond : p-phase condition}$_3$ holds and $H(\cdot,|F|)\in L^\gamma(\Omega_T)$ for some $\gamma\geq 1$, there exists $c> 1$ depending on $n,p,\gamma,|\Omega_T|$, $\|H(z,|Du|)\|_{L^1(\Omega_T)}$ and $\|H(z,|F|)\|_{L^\gamma(\Omega_T)}$ such that
    \begin{equation}\label{eq : decay estimate of rho and lambda}
    \lambda \leq c \rho^{-\frac{n+2}{\mu_2}}.
    \end{equation}
\end{lemma}
\begin{proof}
    By \eqref{cond : p-phase condition}$_3$, we have
    $$
    \begin{aligned}
        \lambda^p &=\miint{Q_\rho^\lambda(z_0)} [H(z,|Du|)+H(z,|F|)]\, dz\\ 
        & = \frac{\lambda^{\frac{(2-p)n}{2}}}{2\rho^{n+2}|B_1|}\iints{Q_\rho^\lambda (z_0)} [H(z,|Du|)+H(z,|F|)]\, dz\\
        & \leq \frac{\|H(z,|Du|)\|_{L^1(\Omega_T)}+\|H(z,|F|)\|_{L^1(\Omega_T)}}{2|B_1|}\cdot\frac{\lambda^{\frac{(2-p)n}{2}}}{\rho^{n+2}}.
    \end{aligned}
    $$ 
    Thus, we obtain
    $$
    \rho\leq \left(\frac{\|H(z,|Du|)\|_{L^1(\Omega_T)}+\|H(z,|F|)\|_{L^1(\Omega_T)}}{2|B_1|}\right)^{\frac{1}{n+2}}\lambda^{-\frac{\mu_2}{n+2}},
    $$
    and hence
    $$
    \lambda \leq c \rho^{-\frac{n+2}{\mu_2}}
    $$
    for some $c=c(n,p,\gamma,|\Omega_T|,\|H(z,|Du|)\|_{L^1(\Omega_T)},\|H(z,|F|)\|_{L^\gamma(\Omega_T)})> 1.$
\end{proof}
The following identity is frequently used in this paper:
\begin{equation}\label{calculate mu_2}
(n+2)(2-p)+2\mu_2=4.
\end{equation}

\textbf{Case \eqref{case : p,q-phase}}: We obtain from \eqref{case : p,q-phase}$_2$ that
$$
4[a]_\alpha (10\varrho_w)^\alpha\leq \sup_{Q_{10\varrho_w}(w)} a(\cdot)\leq \inf_{Q_{10\varrho_w}(w)}a(\cdot)+2[a]_\alpha (10\varrho_w)^\alpha,
$$
and hence
$$
\sup_{Q_{10\varrho_w}(w)}a(\cdot)\leq \inf_{Q_{10\varrho_w} (w)} 2a(\cdot) +[a]_\alpha (10\varrho_w)^\alpha\leq 2\inf_{Q_{10\varrho_w} (w)} a(\cdot).
$$
Therefore, we get
\begin{equation}    \label{cond : comparison in p,q-phase}
    \frac{a(w)}{2}\leq a(\tilde{w})\leq 2a(w) \quad \text{for every } \tilde{w}\in Q_{10\varrho_w} (w).
\end{equation}
Also, by \cite[Subsection 4.1]{Wontae2024}, there exists $\varsigma_w\in (0,\varrho_w]$ such that
\begin{equation}\label{cond : integral of H in =varsigma in p,q-phase}
    \miint{G^{\lambda_w}_{\varsigma_w}(w)} [H(z,|Du|)+H(z,|F|)]\, dz=H_w(\lambda_w)
\end{equation}
and
\begin{equation}\label{cond : integral of H in >varsigma in p,q-phase}
    \miint{G^{\lambda_w}_{\sigma}(w)} [H(z,|Du|)+H(z,|F|)]\, dz<H_w(\lambda_w)
\end{equation}
for any $\sigma\in(\varsigma_w,r_2-r_1)$. Hence, if we replace the center point $w$, radius $\varsigma_w$ and $\lambda_w$ in \eqref{cond : comparison in p,q-phase}-\eqref{cond : integral of H in >varsigma in p,q-phase} with $z_0$, $\rho$ and $\lambda$, respectively, we obtain
\begin{equation}\label{cond : p,q-phase condition}
    \left\{\begin{aligned}
        &K\lambda^p\leq \sup_{Q_{10\rho}(z_0)} a(\cdot)\lambda^q,\quad \frac{a(z_0)}{2}\leq a(z)\leq 2a(z_0)\quad \text{for every }z\in G_{4\rho}^\lambda(z_0),\\
        &\miint{G_\sigma^\lambda (z_0)} [H(z,|Du|)+H(z,|F|)]\, dz <H_{z_0}(\lambda) \quad \text{for any }\sigma\in(\rho,2\kappa\rho],\\
        &\miint{G_\rho^\lambda(z_0)} [H(z,|Du|)+H(z,|F|)]\, dz = H_{z_0}(\lambda).
    \end{aligned}\right.
\end{equation}
\textbf{Case \eqref{case : no occurence}}: We shall rigorously exclude the possibility of this case by proving the estimates
\begin{equation}\label{cond : the impossibility of Case (3)}
\begin{cases}
    \lambda_w \lesssim \varrho_w^{-\frac{4}{p(n+2)-2n}} \qquad&\text{if } \eqref{cond : main assumption with infty} \text{ holds},\\
    \lambda_w \lesssim \varrho_w^{-\frac{n+s}{\mu_s}} &\text{if } \eqref{cond : main assumption with s} \text{ holds}.
\end{cases}
\end{equation}
\begin{lemma}\label{lem : no occurence with s=infty}
    Let $u$ be a weak solution to \eqref{eq : the main equation}, and suppose that
    \begin{equation}\label{cond : no occurence}
    \sup_{Q_{10\varrho_w}(w)}a(\cdot)\leq 4[a]_\alpha (10\varrho_w)^\alpha.
    \end{equation}
    If \eqref{cond : main assumption with infty} and \eqref{cond : source term with infty} hold, then there exists a constant $c_b=c_b(\operatorname{data}_b)>1$ such that
    $$
    \varrho_w\leq c_b\lambda_w^{-\frac{p(n+2)-2n}{4}}.
    $$
\end{lemma}
\begin{proof}
    By Lemma \ref{lem : Caccioppoli inequality} and \eqref{cond : integral of H in =varrho}, we get
    \begin{align}
    \lambda_w^p&=\miint{Q_{\varrho_w}^{\lambda_w}(w)}H(z,|Du|)\,dz \nonumber\\
    &\leq c\miint{Q_{2\varrho_w}^{\lambda_w}(w)}\left(\frac{\Big|u-u_{Q_{2\varrho_w}^{\lambda_w}(w)}\Big|^p}{\left(2\lambda_w^{\frac{p-2}{2}}\varrho_w\right)^p}+a(z)\frac{\Big|u-u_{Q_{2\varrho_w}^{\lambda_w}(w)}\Big|^q}{\left(2\lambda_w^{\frac{p-2}{2}}\varrho_w\right)^q}\right) dz \nonumber\\ 
    &\qquad +c\miint{Q_{2\varrho_w}^{\lambda_w}(w)}\frac{\Big|u-u_{Q_{2\varrho_w}^{\lambda_w}(w)}\Big|^2}{(2\varrho_w)^2}\,dz +c\miint{Q_{2\varrho_w}^{\lambda_w}(w)}H(z,|F|)\,dz\nonumber\\ \label{eq : Caccioppoli ineq in no occurence with infty}
    &=\mathrm{I}_1+\mathrm{I}_2+\mathrm{I}_3+\mathrm{I}_4
    \end{align}
    for some $c=c(n,p,q,\nu,L)>1$. We note from the triangle inequality and Jensen's inequality that
    \begin{equation}\label{eq : estimation of u-mean of u as u}
        \miint{Q_{2\varrho_w}^{\lambda_w}(w)}\Big|u-u_{Q_{2\varrho_w}^{\lambda_w}(w)}\Big|^\gamma\,dz\leq c(\gamma)\miint{Q_{2\varrho_w}^{\lambda_w}(w)}|u|^\gamma \, dz
    \end{equation}
    holds for any $\gamma\in [1,\infty)$.

    \textbf{Estimate of $\mathrm{I}_1$.} By \eqref{eq : estimation of u-mean of u as u}, we get
    $$
    \mathrm{I}_1\leq c\miint{Q_{2\varrho_w}^{\lambda_w}(w)}\frac{|u|^p}{\left(2\lambda_w^{\frac{p-2}{2}}\varrho_w\right)^p} \, dz\leq c\lambda_w^{\frac{(2-p)p}{2}}\varrho_w^{-p}
    $$
    for some $c=c(n,p,q,\nu,L,\|u\|_{L^\infty(\Omega_T)})>1$. Then it follows from \eqref{eq : decay estimate of rho and lambda} and \eqref{calculate mu_2} that
    $$
    \mathrm{I_1}\leq c\varrho_w^{-p\left(\frac{(n+2)(2-p)}{2\mu_2}+1\right)}=c\varrho_w^{-\frac{4p}{p(n+2)-2n}}
    $$
    for some $c>1$ depending on $n,p,q,\nu,L,|\Omega_T|,\|u\|_{L^\infty(\Omega_T)},\|H(z,|Du|)\|_{L^1(\Omega_T)}$ and $\|H(z,|F|)\|_{L^{\gamma_b}(\Omega_T)}$.

    \textbf{Estimate of $\mathrm{I}_2$.} By \eqref{cond : no occurence}, \eqref{eq : estimation of u-mean of u as u} and \eqref{cond : main assumption with infty}$_2$, we get
    $$
    \mathrm{I}_2\leq c\varrho_w^\alpha\miint{Q_{2\varrho_w}^{\lambda_w}(w)}\frac{|u|^q}{\left(2\lambda_w^{\frac{p-2}{2}}\varrho_w\right)^q} \, dz\leq c\lambda_w^{\frac{(2-p)q}{2}} \varrho_w^{\alpha-q}.
    $$
    for some $c=c(n,p,q,\alpha,\nu,L,[a]_\alpha,\|u\|_{L^\infty(\Omega_T)})>1$. Then it follows from \eqref{eq : decay estimate of rho and lambda} and \eqref{calculate mu_2} that
    $$
    \mathrm{I}_2\leq c\varrho_w^{-q\left(\frac{(n+2)(2-p)}{2\mu_2}+1\right)+\alpha}=c\varrho_w^{-\frac{4q}{p(n+2)-2n}+\alpha},
    $$
    where $c>1$ depends on $n,p,q,\alpha,\nu,L,|\Omega_T|,[a]_\alpha,\|u\|_{L^\infty(\Omega_T)},\|H(z,|Du|)\|_{L^1(\Omega_T)}$ and $\|H(z,|F|)\|_{L^{\gamma_b}(\Omega_T)}$. Since \eqref{cond : main assumption with infty} implies 
    $$
    -\frac{4p}{p(n+2)-2n}\leq -\frac{4q}{p(n+2)-2n}+\alpha<0,
    $$
    we have
    $$
    \mathrm{I}_2\leq c\varrho_w^{-\frac{4p}{p(n+2)-2n}}
    $$
    for some $c=c(\operatorname{data}_b)$.

    \textbf{Estimate of $\mathrm{I}_3$.} By \eqref{eq : estimation of u-mean of u as u} and Young's inequality, we get
    $$
    \mathrm{I}_3\leq c\miint{Q_{2\varrho_w}^{\lambda_w}(w)}\frac{|u|^2}{(2\varrho_w)^2} \, dz\leq c\varrho_w^{-2}
    $$
    for some $c=c(n,p,q,\nu,L,\|u\|_{L^\infty(\Omega_T)})>1$. Since $2p(n+2)-4n\leq 4 \leq 4p$ implies $2\leq \frac{4p}{p(n+2)-2n}$, we have 
    $$
    \mathrm{I}_3\leq c\varrho_w^{-\frac{4p}{p(n+2)-2n}}.
    $$

    \textbf{Estimate of $\mathrm{I}_4$.} We obtain from H\"{o}lder's inequality, \eqref{cond : source term with infty} and \eqref{eq : decay estimate of rho and lambda} that
    $$
    \begin{aligned}
        \mathrm{I}_4 \leq c\left(\miint{Q_{2\varrho_w}^{\lambda_w}(w)} [H(z,|F|)]^{\gamma_b}\,dz\right)^\frac{1}{\gamma_b} &\leq c \|H(z,|F|)\|_{L^{\gamma_b}(\Omega_T)}\lambda_w^{\frac{(2-p)n}{2\gamma_b}}\varrho_w^{-\frac{n+2}{\gamma_b}}\\
        &\leq c\varrho_w^{-\left(\frac{(2-p)n+2\mu_2}{\mu_2}\right)}= c \varrho_w^{-\frac{4p}{p(n+2)-2n}},
    \end{aligned}
    $$
    where $c=c(n,p,q,\nu,L,|\Omega_T|,\|H(z,|Du|)\|_{L^1(\Omega_T)},\|H(z,|F|)\|_{L^{\gamma_b}(\Omega_T)})>1$.

    Combining the above results with \eqref{eq : Caccioppoli ineq in no occurence with infty}, we conclude that
    $$
    \lambda_w^p\leq c\varrho_w^{-\frac{4p}{p(n+2)-2n}}
    $$
    for some $c=c(\operatorname{data}_b)>1$.
\end{proof}
Next, we prove \eqref{cond : the impossibility of Case (3)}$_2$ using the Gagliardo-Nirenberg multiplicative embedding inequality.
\begin{lemma}\label{lem : no occurence with s<infty}
Let $u$ be a weak solution to \eqref{eq : the main equation}, and suppose that \eqref{cond : no occurence} is satisfied.
If \eqref{cond : main assumption with s} and \eqref{cond : source term with s} hold for some $2\leq s <\infty$, then there exists a constant $c_s=c_s(\operatorname{data}_s)>1$ such that
$$
\varrho_w\leq c_s\lambda_w^{-\frac{\mu_s}{n+s}}.
$$
\end{lemma}
\begin{proof}
    As in Lemma \ref{lem : no occurence with s=infty}, we infer from Lemma \ref{lem : Caccioppoli inequality} and \eqref{cond : integral of H in =varrho} that
    \begin{align}
        \lambda_w^p&=\miint{Q_{\varrho_w}^{\lambda_w}(w)}H(z,|Du|)\,dz \nonumber\\
        &\leq c\miint{Q_{2\varrho_w}^{\lambda_w}(w)}\left(\frac{\Big|u-u_{Q_{2\varrho_w}^{\lambda_w}(w)}\Big|^p}{\left(2\lambda_w^{\frac{p-2}{2}}\varrho_w\right)^p}+a(z)\frac{\Big|u-u_{Q_{2\varrho_w}^{\lambda_w}(w)}\Big|^q}{\left(2\lambda_w^{\frac{p-2}{2}}\varrho_w\right)^q}\right) dz \nonumber\\ 
        &\qquad +c\miint{Q_{2\varrho_w}^{\lambda_w}(w)}\frac{\Big|u-u_{Q_{2\varrho_w}^{\lambda_w}(w)}\Big|^2}{(2\varrho_w)^2}\,dz+c\miint{Q_{2\varrho_w}^{\lambda_w}(w)}H(z,|F|)\,dz \nonumber\\ \label{eq : Caccioppoli ineq in no occurence}
        &=\mathrm{I}_1+\mathrm{I}_2+\mathrm{I}_3+\mathrm{I}_4
    \end{align}
    for some $c=c(n,p,q,\nu,L)>1$.    
    
    \textbf{Estimate of $\mathrm{I}_3$.} Since $2\leq s$, we see from \eqref{eq : estimation of u-mean of u as u} and H\"{o}lder's inequality that
    \begin{align}
        \mathrm{I}_3&\leq c\miint{Q_{2\varrho_w}^{\lambda_w}(w)}\frac{|u|^2}{(2\varrho_w)^2} \, dz \nonumber\\
        &\leq c\left(\miint{Q_{2\varrho_w}^{\lambda_w}(w)}\frac{|u|^s}{(2\varrho_w)^s} \, dz\right)^{\frac{2}{s}} \nonumber\\
        &\leq c\frac{1}{(2\varrho_w)^2}\left(\frac{1}{|B^{\lambda_w}_{2\varrho_w}|}\sup_{t\in [0,T]}\int_{\Omega}|u|^s\,dx\right)^\frac{2}{s} \nonumber\\
        &\leq c\lambda_w^{\frac{(2-p)n}{s}}\varrho_w^{-\frac{2(n+s)}{s}} \nonumber\\ 
        &\leq c\lambda_w^{\frac{(2-p)n}{2}}\varrho_w^{-\frac{2(n+s)}{s}} \nonumber
    \end{align}
    for some $c=c(n,p,q,s,\nu,L,\|u\|_{C(0,T;L^s(\Omega))})>1$. Note that 
    $$
    \begin{aligned}
        \frac{2n}{n+2}<p \quad &\implies \quad 2n < pn + 2p\\
        &\implies \quad -2p<pn-2n\\
        &\implies \quad 0=2p-2p<2p-2n+pn=2p-(2-p)n\\
        &\implies \quad \frac{2p}{(2-p)n}>1.
    \end{aligned}
    $$
    Thus, applying Young's inequality with the exponents $\frac{2p}{(2-p)n}$ and $\frac{2p}{2p-2n+pn}$, we have
    \begin{equation}\label{eq : estimate of I_3 in lem_no occurence}
    \mathrm{I}_3\leq \frac{1}{2}\lambda_w^{p}+c\varrho_w^{-\frac{p(n+s)}{\mu_s}}
    \end{equation}
    for some $c=c(n,p,q,s,\nu,L,\|u\|_{C(0,T;L^s(\Omega))})>1$.

    \textbf{Estimate of $\mathrm{I}_1$.} Since $p\leq 2 \leq s$ and $2\mu_s=s\mu_2$, we deduce from \eqref{eq : decay estimate of rho and lambda}, \eqref{calculate mu_2}, \eqref{eq : estimation of u-mean of u as u} and H\"{o}lder's inequality that 
    \begin{align}
        \mathrm{I}_1&\leq c\miint{Q_{2\varrho_w}^{\lambda_w}(w)}\frac{|u|^p}{\left(2\lambda_w^{\frac{p-2}{2}}\varrho_w\right)^p} \, dz \nonumber\\
        &\leq c\lambda_w^{\frac{(2-p)p}{2}}\left(\miint{Q_{2\varrho_w}^{\lambda_w}(w)}\frac{|u|^s}{(2\varrho_w)^s} \, dz\right)^{\frac{p}{s}} \nonumber\\ 
        &\leq c\lambda_w^{\frac{(2-p)p}{2}} \varrho_w^{-p}\left(\frac{1}{|B^{\lambda_w}_{2\varrho_w}|}\sup_{t\in [0,T]}\int_{\Omega}|u|^s\,dx\right)^\frac{p}{s} \nonumber \\ 
        &\leq c \lambda_w^{\frac{p(2-p)(n+s)}{2s}}\varrho_w^{-\frac{p(n+s)}{s}}\nonumber\\ 
        &\leq c \varrho_w^{-\frac{p(n+s)((2-p)(n+2)+2\mu_2)}{2s\mu_2}}\nonumber\\ \label{eq : estimate of I_1 in lem_no occurence}
        &= c \varrho_w^{-\frac{p(n+s)}{\mu_s}}
    \end{align}
    for some $c>1$ depending on $n,p,q,s,\nu,L,|\Omega_T|,\|u\|_{C(0,T;L^s(\Omega))},\|H(z,|Du|)\|_{L^1(\Omega_T)}$ and $\|H(z,|F|)\|_{L^{\gamma_s}(\Omega_T)}$.

    \textbf{Estimate of $\mathrm{I}_2$.} By \eqref{cond : no occurence}, we have
    $$
    \mathrm{I}_2\leq c\varrho_w^\alpha \miint{Q_{2\varrho_w}^{\lambda_w}(w)}\frac{\Big|u-u_{Q_{2\varrho_w}^{\lambda_w}(w)}\Big|^q}{\left(2\lambda_w^{\frac{p-2}{2}}\varrho_w\right)^q}\,dz
    $$
    for some $c=c(n,p,q,\alpha,\nu,L,[a]_\alpha)>1$. We divide the cases according to $q$ and $s$.

    If $q\leq s$, since \eqref{cond : main assumption with s}$_2$ implies that $-\frac{p(n+s)}{\mu_s}\leq \alpha -\frac{q(n+s)}{\mu_s}<0$, it follows from \eqref{eq : decay estimate of rho and lambda}, \eqref{calculate mu_2}, \eqref{eq : estimation of u-mean of u as u} and H\"{o}lder's inequality that
    \begin{align}
        \mathrm{I}_2&\leq c\varrho_w^\alpha \left(\miint{Q_{2\varrho_w}^{\lambda_w}(w)}\frac{\Big|u-u_{Q_{2\varrho_w}^{\lambda_w}(w)}\Big|^s}{\left(2\lambda_w^{\frac{p-2}{2}}\varrho_w\right)^s}\,dz\right)^{\frac{q}{s}}\nonumber\\
        &\leq c\lambda_w^{\frac{(2-p)(n+s)q}{2s}}\varrho_w^{\alpha-\frac{q(n+s)}{s}} \nonumber\\
        &\leq c\varrho_w^{\alpha-\frac{q(n+s)((2-p)(n+2)+2\mu_2)}{2s\mu_2}} \nonumber\\
        &\leq c\varrho_w^{\alpha-\frac{q(n+s)}{\mu_s}} \nonumber\\ \label{eq : estimate of I_2 in lem_no occurence with q<=s}
        &\leq c\varrho_w^{-\frac{p(n+s)}{\mu_s}}
    \end{align}
    for some $c=c(\operatorname{data}_s)>1$.

    Finally, assume that $q>s$. Then we obtain
    \begin{align*}
        \mathrm{I}_2&\leq c\varrho_w^\alpha\miint{Q_{2\varrho_w}^{\lambda_w}(w)}\frac{\Big|u-u_{B^{\lambda_w}_{2\varrho_w}(x_0)}(t)\Big|^q}{\left(2\lambda_w^{\frac{p-2}{2}}\varrho_w\right)^q}\,dx dt\\
        &\qquad+c\varrho_w^\alpha\dashint_{I_{2\varrho_w}(t_0)}\frac{\Big|u_{B^{\lambda_w}_{2\varrho_w}(x_0)}(t)-u_{Q_{2\varrho_w}^{\lambda_w}(w)}\Big|^q}{\left(2\lambda_w^{\frac{p-2}{2}}\varrho_w\right)^q}\, dt \\ 
        &=\mathrm{J}_1+\mathrm{J}_2,
    \end{align*}
    where $w=(x_0, t_0)$. By the Gagliardo-Nirenberg multiplicative embedding inequality in \cite[Theorem 2.1 and Remark 2.1 in Section I]{1993_Degenerate_parabolic_equations_DiBenedetto}, we get
    $$
    \begin{aligned}
        \mathrm{J}_1&= c\varrho_w^\alpha\dashint_{I_{2\varrho_w}(t_0)}\left(\dashint_{B^{\lambda_w}_{2\varrho_w}(x_0)}\frac{\Big|u-u_{B_{2\varrho_w}(x_0)}(t)\Big|^q}{\left(2\lambda_w^{\frac{p-2}{2}}\varrho_w\right)^q}\,dx\right)\,dt\\
        &\leq c\varrho_w^\alpha\dashint_{I_{2\varrho_w}(t_0)}\left(\dashint_{B^{\lambda_w}_{2\varrho_w}(x_0)} |Du|^p\, dx\right)^{\frac{q\theta_1}{p}}\\
        &\qquad\qquad \times \left(\dashint_{B^{\lambda_w}_{2\varrho_w}(x_0)}\frac{\Big|u-u_{B^{\lambda_w}_{2\varrho_w}(x_0)}(t)\Big|^s}{\left(2\lambda_w^{\frac{p-2}{2}}\varrho_w\right)^s}\,dx\right)^{\frac{q(1-\theta_1)}{s}}\,dt,
    \end{aligned}
    $$
    where $\theta_1=\left(\frac{1}{s}-\frac{1}{q}\right)\left(\frac{1}{n}+\frac{1}{s}-\frac{1}{p}\right)^{-1}$ and $c=c(n,p,q,s,\nu,L)>1$. Since $s<q\leq 4$, $\frac{2n}{n+2}<p\leq 2\leq n$, we observe that $\theta_1$ is in $[0,1]$. Now, by \eqref{eq : decay estimate of rho and lambda}, \eqref{calculate mu_2}, \eqref{cond : p-phase condition}$_3$ and H\"{o}lder's inequality, we have
    \begin{equation*}
    \begin{aligned}
        \mathrm{J}_1&\leq c\lambda_w^{\frac{q(n+s)(2-p)(1-\theta_1)}{2s}}\varrho_w^{\alpha-\frac{q(n+s)(1-\theta_1)}{s}}\left(\miint{Q_{2\varrho_w}^{\lambda_w}(w)} |Du|^p\, dz\right)^{\frac{q\theta_1}{p}} \\
        &\leq c\varrho_w^{\alpha-\frac{q(n+s)(1-\theta_1)}{s}}\lambda_w^{\frac{q(n+s)(2-p)(1-\theta_1)}{2s}+q\theta_1} \\
        &\leq c\varrho_w^{\alpha-\frac{q(n+s)(1-\theta_1)(2\mu_2+(2-p)(n+2))}{2s\mu_2}-\frac{q\theta_1(n+2)}{\mu_2}}=c\varrho_w^{\alpha-\frac{q(n+s)(1-\theta_1)}{\mu_s}-\frac{qs\theta_1(n+2)}{2\mu_s}}\\
        &= c\varrho_w^{\alpha-\frac{2q(n+s)-2q\theta_1(n+s)+qs\theta_1(n+2)}{2\mu_s}}= c\varrho_w^{\alpha-\frac{2q(n+s)-qn\theta_1(s-2)}{2\mu_s}}\leq c\varrho_w^{\alpha-\frac{q(n+s)}{\mu_s}}
    \end{aligned}
    \end{equation*}
    for some $c=c(\operatorname{data}_s)>1$. Since $-\frac{p(n+s)}{\mu_s}\leq \alpha -\frac{q(n+s)}{\mu_s}<0$, we get
    \begin{equation*}
        \mathrm{J}_1\leq c\varrho_w^{-\frac{p(n+s)}{\mu_s}}.
    \end{equation*}
    Next, \eqref{eq : decay estimate of rho and lambda}, \eqref{calculate mu_2}, \eqref{cond : main assumption with s} and H\"{o}lder's inequality imply that
    $$
    \begin{aligned}
        \mathrm{J}_2&\leq c\lambda_w^{\frac{(2-p)q}{2}}\varrho_w^{\alpha-q}\dashint_{I_{2\varrho_w}(t_0)}\dashint_{I_{2\varrho_w}(t_0)}|u_{B^{\lambda_w}_{2\varrho_w}(x_0)}(t)-u_{B^{\lambda_w}_{2\varrho_w}(x_0)}(\tilde{t})|^q\, dt d\tilde{t}\\
        &\leq c\lambda_w^{\frac{(2-p)q}{2}}\varrho_w^{\alpha-q} \dashint_{I_{2\varrho_w}(t_0)}|u_{B^{\lambda_w}_{2\varrho_w}(x_0)}(t)|^q\, dt\\
        &\leq c\lambda_w^{\frac{(2-p)q}{2}}\varrho_w^{\alpha-q}\sup_{I_{2\varrho_w}(t_0)}\left(\dashint_{B^{\lambda_w}_{2\varrho_w}(x_0)}|u|\, dx\right)^q\\
        &\leq c\lambda_w^{\frac{(2-p)q}{2}}\varrho_w^{\alpha-q}\sup_{I_{2\varrho_w}(t_0)}\left(\dashint_{B^{\lambda_w}_{2\varrho_w}(x_0)}|u|^s\, dx\right)^\frac{q}{s}\\
        &\leq c\lambda_w^{\frac{(2-p)q(n+s)}{2s}}\varrho_w^{\alpha-\frac{q(n+s)}{s}}\leq c \varrho_w^{\alpha-\frac{q(n+s)((n+2)(2-p)+2\mu_2)}{2s\mu_2}}\\
        &=c\varrho_w^{\alpha-\frac{q(n+s)}{\mu_s}}\leq c\varrho_w^{-\frac{p(n+s)}{\mu_s}}
    \end{aligned}
    $$
    for some $c=c(\operatorname{data}_s)>1$.
    Thus, we obtain
    \begin{equation}\label{eq : estimate of I_2 in lem_no occurence with q>s}
        \mathrm{I}_2\leq c \varrho_w^{-\frac{p(n+s)}{\mu_s}}.
    \end{equation}
    We then conclude from \eqref{eq : estimate of I_2 in lem_no occurence with q<=s} and \eqref{eq : estimate of I_2 in lem_no occurence with q>s} that
    \begin{equation}\label{eq : estimate of I_2 in lem_no occurence}
        \mathrm{I}_2\leq c\varrho_w^{-\frac{p(n+s)}{\mu_s}},
    \end{equation}
    where $c=c(\operatorname{data}_s)>1$.

    \textbf{Estimate of $\mathrm{I}_4$.} We obtain from H\"{o}lder's inequality, \eqref{cond : source term with s} and \eqref{eq : decay estimate of rho and lambda} that
    \begin{equation}\label{eq : estimate of I_4 in lem_no occurence}
    \begin{aligned}
        \mathrm{I}_4 \leq c\left(\miint{Q_{2\varrho_w}^{\lambda_w}(w)} [H(z,|F|)]^{\gamma_s}\,dz\right)^\frac{1}{\gamma_s} &\leq c \|H(z,|F|)\|_{L^{\gamma_s}(\Omega_T)}\lambda_w^{\frac{(2-p)n}{2\gamma_s}}\varrho_w^{-\frac{n+2}{\gamma_s}}\\
        &\leq c\varrho_w^{-\frac{2p(n+s)}{s\mu_2}}= c \varrho_w^{-\frac{p(n+s)}{\mu_s}},
    \end{aligned}
    \end{equation}
    where $c=c(n,p,q,s,\nu,L,|\Omega_T|,\|H(z,|Du|)\|_{L^1(\Omega_T)},\|H(z,|F|)\|_{L^{\gamma_s}(\Omega_T)})>1$.

    Combining \eqref{eq : Caccioppoli ineq in no occurence}, \eqref{eq : estimate of I_3 in lem_no occurence}, \eqref{eq : estimate of I_1 in lem_no occurence}, \eqref{eq : estimate of I_2 in lem_no occurence} and \eqref{eq : estimate of I_4 in lem_no occurence} gives
    $$
    \lambda_w^p\leq c\varrho_w^{-\frac{p(n+s)}{\mu_s}},
    $$
    which completes the proof.
\end{proof}

Now, we show that the case \eqref{case : no occurence} never occurs. If \eqref{case : no occurence} holds, we have
$$
K\lambda_w^p=\sup_{Q_{10\varrho_w}(w)} a(\cdot)\frac{K\lambda_w^p}{\displaystyle\sup_{Q_{10\varrho_w}(w)} a(\cdot)}\leq 40[a]_\alpha\varrho_w^\alpha \lambda_w^q.
$$
When \eqref{cond : main assumption with infty} holds, then it follows from Lemma \ref{lem : no occurence with s=infty} and \eqref{def : definition of K and kappa} that
$$
K\lambda_w^p\leq 40[a]_\alpha\varrho_w^\alpha \lambda_w^q\leq 40c_b[a]_\alpha\lambda_w^{q-\frac{\alpha(p(n+2)-2n)}{4}}\leq 40c_b[a]_\alpha\lambda_w^{p}<\frac{K}{2}\lambda_w^p,
$$
which is a contradiction. Similarly, when \eqref{cond : main assumption with s} holds, then it follows from Lemma \ref{lem : no occurence with s<infty} and \eqref{def : definition of K and kappa} that
$$
K\lambda_w^p\leq 40[a]_\alpha\varrho_w^\alpha \lambda_w^q\leq 40c_s[a]_\alpha\lambda_w^{q-\frac{\alpha\mu_s}{n+s}}\leq 40c_s[a]_\alpha\lambda_w^{p}<\frac{K}{2}\lambda_w^p,
$$
which is a contradiction. Thus, the case \eqref{case : no occurence} can never happen under either \eqref{cond : main assumption with infty} or \eqref{cond : main assumption with s}.
\section{\bf Reverse H\"{o}lder inequality}\label{section 4}
Let $z_0=(x_0,t_0)\in\Psi(\Lambda)$ be a Lebesgue point of $|Du(z)|^p+a(z)|Du(z)|^q$, where $\Lambda$ is defined in Section \ref{section 3}. In this section, we establish reverse H\"{o}lder inequalities separately in each intrinsic cylinder. For this, we need the following auxiliary lemmas, called the Gagliardo-Nirenberg inequality and a standard iteration lemma.
\begin{lemma}[\cite{Hasto_2021}, Lemma 2.12]   \label{lem : Gagliardo-Nirenberg inequality}
    For an open ball $B_{\rho}(x_0)\subset \mr^n$, take $p_1,\,p_2,\,p_3\in[1,\infty)$, $\vartheta\in(0,1)$ and let $\psi\in W^{1,p_2}(B)$. Suppose that
    $$
    -\frac{n}{p_1}\leq \vartheta\left(1-\frac{n}{p_2}\right)-(1-\vartheta)\frac{n}{p_3}.
    $$
    Then there exists a positive constant $c=c(n,p_1)$ such that
    $$
    \dashint_B \frac{|\psi|^{p_1}}{\rho^{p_1}}\, dx\leq c \left(\dashint_{B_\rho(x_0)}\left[\frac{|\psi|^{p_2}}{\rho^{p_2}}+|D\psi|^{p_2}\right]dx\right)^{\frac{\vartheta p_1}{p_2}}\left(\dashint_{B_{\rho}(x_0)}\frac{|\psi|^{p_3}}{\rho^{p_3}}\,dx\right)^{\frac{(1-\vartheta)p_1}{p_3}}.
    $$
\end{lemma}
\begin{lemma}[\cite{2003_Giusti_Direct_methods_in_the_calculus_of_variations}, Lemma 6.1]   \label{lem : a standard iteration lemma}
    Let $0<\rho<\tau<\infty$, and let $g:[\rho,\tau]\rightarrow [0,\infty)$ be a bounded function. Suppose that
    $$
    g(\rho_1)\leq \vartheta g(\rho_2)+\frac{A}{(\rho_2-\rho_1)^\gamma}+B
    $$
    holds for all $0<\rho\leq \rho_1<\rho_2\leq \tau$, where $\vartheta \in (0,1)$, $A,B\geq 0$ and $\gamma>0$. Then there exists a positive constant $c$ depending on $\vartheta$ and $\gamma$ such that
    $$
    g(\rho)\leq c\left(\frac{A}{(\tau-\rho)^{\gamma}}+B\right).
    $$
\end{lemma}
\subsection{\bf The $p$-phase case}
We assume \eqref{cond : p-phase condition} and estimate the last term in Lemma \ref{lem : semi-Parabolic Poincare inequality}.
\begin{lemma}   \label{lem : last term estimate in Lemma lem : semi-Parabolic Poincare inequality whenever p-intrinsic cylinder}
    Let $u$ be a weak solution to \eqref{eq : the main equation} and assume that $Q_{4\rho}^\lambda(z_0)\subset \Omega_T$ satisfies \eqref{cond : p-phase condition}. Then, for $\sigma\in[2\rho,4\rho]$, there exists a constant $c=c(\operatorname{data})>1$ such that
    $$
    \begin{aligned}
        &\miint{Q_\sigma^\lambda(z_0)} \left(|Du|^{p-1}+a(z)|Du|^{q-1}+|F|^{p-1}+a(z)|F|^{q-1}\right) dz\\
        &\qquad\leq c\miint{Q_\sigma^\lambda(z_0)} (|Du|+|F|)^{p-1}\, dz+c\lambda^{-1+\frac{p}{q}}\miint{Q_\sigma^\lambda(z_0)} a(z)^{\frac{q-1}{q}}(|Du|+|F|)^{q-1}\, dz.\\
    \end{aligned}
    $$
\end{lemma}
\begin{proof}
    By \eqref{cond : p-phase condition}$_1$, there exists a constant $c=c(\operatorname{data})>1$ such that
    \begin{align*}
        &\miint{Q_\sigma^\lambda(z_0)} \left(|Du|^{p-1}+a(z)|Du|^{q-1}+|F|^{p-1}+a(z)|F|^{q-1}\right) dz\\
        &\qquad\leq \miint{Q_\sigma^\lambda(z_0)} \left(|Du|^{p-1}+|F|^{p-1}\right)\, dz\\
        &\qquad\qquad +\sup_{w\in Q_{10\rho}(z_0)} a(w)^{\frac{1}{q}}\miint{Q_\sigma^\lambda(z_0)} a(z)^{\frac{q-1}{q}}(|Du|^{q-1}+|F|^{q-1})\, dz\\
        &\quad\quad\leq c\miint{Q_\sigma^\lambda(z_0)} (|Du|+|F|)^{p-1}\, dz+ c\lambda^{-1+\frac{p}{q}}\miint{Q_\sigma^\lambda(z_0)} a(z)^{\frac{q-1}{q}}(|Du|+|F|)^{q-1}\, dz.
    \end{align*}
\end{proof}
\subsubsection{Assumption \eqref{cond : main assumption with infty}}   Now, let $u$ be a weak solution to \eqref{eq : the main equation} and assume that $Q_{4\rho}^\lambda(z_0)\subset \Omega_T$ satisfies \eqref{cond : p-phase condition}. Moreover, we assume \eqref{cond : main assumption with infty}. First, we establish a $p$-intrinsic parabolic Poincar\'{e} inequality.
    \begin{lemma}\label{lem : p-intrinsic parabolic Poincare inequality of p-term in p-intrinsic cylinder with bounded solution}
    For $\sigma\in[2\rho,4\rho]$ and $\theta \in \left(\frac{q-1}{p},1\right]$, there exists a constant $c=c(\operatorname{data}_b)$ $>1$ such that
    \begin{align*}
        &\miint{Q_\sigma^\lambda(z_0)}\frac{\Big|u-u_{Q_\sigma^\lambda(z_0)}\Big|^{\theta p}}{(\lambda^{\frac{p-2}{2}}\sigma)^{\theta p}}\, dz\\
        &\qquad\qquad\leq c\miint{Q_\sigma^\lambda(z_0)} [H(z,|Du|)]^\theta\, dz\\
        &\qquad\qquad\quad +c\lambda^{\left(2-p+\frac{\alpha(p(n+2)-2n)}{8}\right)\theta p}\left(\miint{Q_\sigma^\lambda(z_0)} (|Du|+|F|)^{\theta p}\, dz\right)^{p-1-\frac{\alpha(p(n+2)-2n)}{8}}.
    \end{align*}
\end{lemma}
\begin{proof}
    By Lemmas \ref{lem : semi-Parabolic Poincare inequality} and \ref{lem : last term estimate in Lemma lem : semi-Parabolic Poincare inequality whenever p-intrinsic cylinder}, there exists a positive constant $c=c(\operatorname{data}_b)$ such that 
    \begin{align*}
        \miint{Q_\sigma^\lambda(z_0)}\frac{\Big|u-u_{Q_\sigma^\lambda(z_0)}\Big|^{\theta p}}{(\lambda^{\frac{p-2}{2}}\sigma)^{\theta p}}\, dz &\leq  c \miint{Q_\sigma^\lambda(z_0)}|D u|^{\theta p} \,dz\\
        &\quad +c\left(\lambda^{2-p}\miint{Q_\sigma^\lambda(z_0)} (|Du|+|F|)^{p-1}\, dz\right)^{\theta p}\\
        &\quad +c\left(\lambda^{1-p+\frac{p}{q}}\miint{Q_\sigma^\lambda(z_0)} a(z)^{\frac{q-1}{q}} (|Du|+|F|)^{q-1}\, dz\right)^{\theta p}.
    \end{align*}
    Note that 
    \begin{align*}
        p-1-\frac{\alpha(p(n+2)-2n)}{8}&> p-1 -\frac{\alpha(p(n+2)-2n)}{4}\\
        &\geq p-1-\frac{p(n+2)-2n}{4}\\
        &= \frac{(2-p)(n-2)}{4}\geq 0.
    \end{align*}
    Then it follows from \eqref{cond : p-phase condition} and H\"{o}lder's inequality that
    $$
    \begin{aligned}
        &\lambda^{(2-p)\theta p}\left(\miint{Q_\sigma^\lambda(z_0)} (|Du|+|F|)^{p-1}\, dz\right)^{\theta p}\\
        &\;  \leq \lambda^{(2-p)\theta p}\left(\miint{Q_\sigma^\lambda(z_0)} (|Du|+|F|)^{\theta p}\, dz\right)^{p-1}\\
        &\; =  \lambda^{(2-p)\theta p} \left(\miint{Q_\sigma^\lambda(z_0)} (|Du|+|F|)^{\theta p}\, dz\right)^{\frac{\alpha(p(n+2)-2n)}{8}}\\
        &\quad \times\left(\miint{Q_\sigma^\lambda(z_0)} (|Du|+|F|)^{\theta p}\, dz\right)^{p-1-\frac{\alpha(p(n+2)-2n)}{8}}\\
        &\; \leq \lambda^{\left(2-p+\frac{\alpha(p(n+2)-2n)}{8}\right)\theta p}\left(\miint{Q_\sigma^\lambda(z_0)} (|Du|+|F|)^{\theta p}\, dz\right)^{p-1-\frac{\alpha(p(n+2)-2n)}{8}}.
    \end{aligned}
    $$
    Next, using \eqref{cond : p-phase condition} and H\"{o}lder's inequality, we have
    $$
    \begin{aligned}
        &\left(\lambda^{1-p+\frac{p}{q}}\miint{Q_\sigma^\lambda(z_0)} a(z)^{\frac{q-1}{q}}(|Du|+|F|)^{q-1}\, dz\right)^{\theta p}\\
        &\qquad \leq c\lambda^{\left(1-p+\frac{p}{q}+\frac{(p-q)(q-1)}{q}\right)\theta p}\left(\miint{Q_\sigma^\lambda(z_0)} (|Du|+|F|)^{q-1}\, dz\right)^{\theta p}\\
        &\qquad \leq c\lambda^{\left(2-p+\frac{\alpha(p(n+2)-2n)}{8}\right)\theta p}\left(\miint{Q_\sigma^\lambda(z_0)} (|Du|+|F|)^{\theta p}\, dz\right)^{p-1-\frac{\alpha(p(n+2)-2n)}{8}}
    \end{aligned}
    $$
    for some $c=c(\operatorname{data}_b)>1$. This completes the proof.
\end{proof}
\begin{lemma}\label{lem : p-intrinsic parabolic Poincare inequality of q-term in p-intrinsic cylinder with bounded solution}
    For $\sigma\in[2\rho,4\rho]$ and $\theta \in \left(\frac{q-1}{p},1\right]$, there exists a constant $c=c(\operatorname{data}_b)$ $>1$ such that
    \begin{align*}
        &\miint{Q_\sigma^\lambda(z_0)}\inf_{w\in Q_\sigma^\lambda (z_0)} a(w)^\theta \frac{\Big|u-u_{Q_\sigma^\lambda(z_0)}\Big|^{\theta q}}{(\lambda^{\frac{p-2}{2}}\sigma)^{\theta q}}\, dz\\
        &\qquad\qquad\leq c\miint{Q_\sigma^\lambda(z_0)} [H(z,|Du|)]^\theta\, dz\\
        &\qquad\qquad\quad +c\lambda^{\left(2-p+\frac{\alpha(p(n+2)-2n)}{8}\right)\theta p}\left(\miint{Q_\sigma^\lambda(z_0)} (|Du|+|F|)^{\theta p}\, dz\right)^{p-1-\frac{\alpha(p(n+2)-2n)}{8}}.
    \end{align*}
\end{lemma}
\begin{proof}
    By Lemmas \ref{lem : semi-Parabolic Poincare inequality} and \ref{lem : last term estimate in Lemma lem : semi-Parabolic Poincare inequality whenever p-intrinsic cylinder}, there exists a constant $c=c(\operatorname{data}_b)>1$ such that 
    \begin{align*}
        &\miint{Q_\sigma^\lambda(z_0)}\inf_{w\in Q_\sigma^\lambda (z_0)} a(w)^\theta \frac{\Big|u-u_{Q_\sigma^\lambda(z_0)}\Big|^{\theta q}}{(\lambda^{\frac{p-2}{2}}\sigma)^{\theta q}}\, dz\\
        &\qquad\qquad\leq  c \miint{Q_\sigma^\lambda(z_0)} \inf_{w\in Q_\sigma^\lambda (z_0)} a(w)^\theta|D u|^{\theta q} \,dz\\
        &\qquad\qquad\quad +c\inf_{w\in Q_\sigma^\lambda (z_0)} a(w)^\theta\left(\lambda^{2-p}\miint{Q_\sigma^\lambda(z_0)} (|Du|+|F|)^{p-1}\, dz\right)^{\theta q}\\
        &\qquad\qquad\quad +c\inf_{w\in Q_\sigma^\lambda (z_0)} a(w)^\theta\left(\lambda^{1-p+\frac{p}{q}}\miint{Q_\sigma^\lambda(z_0)} a(z)^{\frac{q-1}{q}} (|Du|+|F|)^{q-1}\, dz\right)^{\theta q}.
    \end{align*}
    By \eqref{cond : p-phase condition} and H\"{o}lder's inequality, the second term on the right-hand side is estimated by
    $$
    \begin{aligned}
        &\inf_{w\in Q_\sigma^\lambda (z_0)} a(w)^\theta\left(\lambda^{2-p}\miint{Q_\sigma^\lambda(z_0)} (|Du|+|F|)^{p-1}\, dz\right)^{\theta q}\\
        &\; \leq \inf_{w\in Q_\sigma^\lambda (z_0)} a(w)^\theta\lambda^{(2-p)\theta q}\left(\miint{Q_\sigma^\lambda(z_0)} (|Du|+|F|)^{\theta p}\, dz\right)^{\frac{q(p-1)}{p}}\\
        &\; \leq  c\lambda^{\theta p\left[1-\frac{q}{p}+(2-p)\frac{q}{p}+(p-1)(\frac{q}{p}-1)+\frac{\alpha(p(n+2)-2n)}{8}\right]}\\
        &\qquad\qquad\qquad\qquad \times\left(\miint{Q_\sigma^\lambda(z_0)} (|Du|+|F|)^{\theta p}\, dz\right)^{p-1-\frac{\alpha(p(n+2)-2n)}{8}}\\
        &\;= c\lambda^{\theta p \left(2-p+\frac{\alpha(p(n+2)-2n)}{8}\right)} \left(\miint{Q_\sigma^\lambda(z_0)} (|Du|+|F|)^{\theta p}\, dz\right)^{p-1-\frac{\alpha(p(n+2)-2n)}{8}}
    \end{aligned}
    $$
    for some $c=c(\operatorname{data}_b)> 1$. Similarly, the last term on the right-hand side is estimated by
    $$
    \begin{aligned}
        &\inf_{w\in Q_\sigma^\lambda (z_0)} a(w)^\theta\left(\lambda^{1-p+\frac{p}{q}}\miint{Q_\sigma^\lambda(z_0)} a(z)^{\frac{q-1}{q}}(|Du|+|F|)^{q-1}\, dz\right)^{\theta q}\\
        &\qquad\qquad \leq c\lambda^{\theta p \left(2-p+\frac{\alpha(p(n+2)-2n)}{8}\right)} \left(\miint{Q_\sigma^\lambda(z_0)} (|Du|+|F|)^{\theta p}\, dz\right)^{p-1-\frac{\alpha(p(n+2)-2n)}{8}}
    \end{aligned}
    $$
    for some $c=c(\operatorname{data}_b)> 1$.
\end{proof}
Next, we consider 
$$
S(u,Q_\rho^\lambda(z_0))\coloneq \sup_{I_\rho (t_0)}\dashint_{B_\rho^\lambda(x_0)} \frac{\Big|u-u_{Q_\rho^\lambda (z_0)}\Big|^2}{\left(\lambda^{\frac{p-2}{2}}\rho\right)^2}\, dx.
$$
\begin{lemma}\label{lem : estimate of S in p-intrinsic cylinder of bounded solution}
    There exists a constant $c=c(\operatorname{data}_b)>1$ such that
    $$
    S(u,Q_{2\rho}^\lambda (z_0))=\sup_{I_{2\rho}(t_0)}\dashint_{B_{2\rho}^\lambda(x_0)}\frac{\Big|u-u_{Q_{2\rho}^\lambda (z_0)}\Big|^2}{\left(2\lambda^{\frac{p-2}{2}}\rho\right)^2}\, dx \leq c\lambda^2.
    $$
\end{lemma}
\begin{proof}
    Let $2\rho\leq \rho_1<\rho_2\leq 4\rho$. By Lemma \ref{lem : Caccioppoli inequality}, there exists a constant $c=c(n,p,q,\nu,L)>1$ such that
    \begin{align*}
        &\lambda^{p-2}S(u,Q_{\rho_1}^\lambda (z_0)) \\
        &\qquad\leq \frac{c\rho_2^q}{(\rho_2-\rho_1)^q}\miint{Q_{\rho_2}^\lambda (z_0)}\left(\frac{\Big|u-u_{Q_{\rho_2}^\lambda (z_0)}\Big|^p}{\left(\lambda^{\frac{p-2}{2}}\rho_2\right)^p}+a(z)\frac{\Big|u-u_{Q_{\rho_2}^\lambda(z_0)}\Big|^q}{\left(\lambda^{\frac{p-2}{2}}\rho_2\right)^q}\right)\, dz \\
        &\qquad\quad+\frac{c\rho_2^2}{(\rho_2-\rho_1)^2}\miint{Q_{\rho_2}^\lambda (z_0)}\frac{\Big|u-u_{Q_{\rho_2}^\lambda(z_0)}\Big|^2}{\rho_2^2}\, dz+c\miint{Q_{\rho_2}^\lambda(z_0)} H(z,|F|)\, dz.
    \end{align*}
    By \eqref{cond : p-phase condition}$_2$ and Lemma \ref{lem : p-intrinsic parabolic Poincare inequality of p-term in p-intrinsic cylinder with bounded solution}, we obtain
    \begin{equation*}
    \miint{Q_{\rho_2}^\lambda (z_0)}\frac{\Big|u-u_{Q_{\rho_2}^\lambda (z_0)}\Big|^p}{\left(\lambda^{\frac{p-2}{2}}\rho_2\right)^p}\, dz \leq c\lambda^p
    \end{equation*}
    for some $c=c(\operatorname{data}_b)>1$. On the other hand, we have
    $$
    \begin{aligned}
        \miint{Q_{\rho_2}^\lambda (z_0)}a(z)\frac{\Big|u-u_{Q_{\rho_2}^\lambda(z_0)}\Big|^q}{\left(\lambda^{\frac{p-2}{2}}\rho_2\right)^q}\, dz &\leq \miint{Q_{\rho_2}^\lambda (z_0)} \inf_{w\in Q_{\rho_2}^\lambda(z_0)} a(w)\frac{\Big|u-u_{Q_{\rho_2}^\lambda(z_0)}\Big|^q}{\left(\lambda^{\frac{p-2}{2}}\rho_2\right)^q}\, dz\\
        &\qquad +[a]_\alpha \rho_2^\alpha\miint{Q_{\rho_2}^\lambda (z_0)}\frac{\Big|u-u_{Q_{\rho_2}^\lambda(z_0)}\Big|^q}{\left(\lambda^{\frac{p-2}{2}}\rho_2\right)^q}\, dz.
    \end{aligned}
    $$
    Using \eqref{cond : p-phase condition}$_2$ and Lemma \ref{lem : p-intrinsic parabolic Poincare inequality of q-term in p-intrinsic cylinder with bounded solution} gives
    $$
    \miint{Q_{\rho_2}^\lambda (z_0)} \inf_{w\in Q_{\rho_2}^\lambda(z_0)} a(w)\frac{\Big|u-u_{Q_{\rho_2}^\lambda(z_0)}\Big|^q}{\left(\lambda^{\frac{p-2}{2}}\rho_2\right)^q}\, dz\leq c\lambda^p
    $$
    for some $c=c(\operatorname{data}_b)>1$. Furthermore, since $u\in L^\infty(\Omega_T)$, it follows from \eqref{cond : main assumption with infty}, \eqref{eq : decay estimate of rho and lambda} and \eqref{calculate mu_2} that
    $$
    \begin{aligned}
        &\rho_2^\alpha\miint{Q_{\rho_2}^\lambda (z_0)}\frac{\Big|u-u_{Q_{\rho_2}^\lambda(z_0)}\Big|^q}{\left(\lambda^{\frac{p-2}{2}}\rho_2\right)^q}\, dz\\
        &\qquad\qquad = \rho_2^\alpha \miint{Q_{\rho_2}^\lambda (z_0)} \frac{\Big|u-u_{Q_{\rho_2}^\lambda(z_0)}\Big|^{q-p}}{\left(\lambda^{\frac{p-2}{2}}\rho_2\right)^{q-p}}\cdot \frac{\Big|u-u_{Q_{\rho_2}^\lambda(z_0)}\Big|^{p}}{\left(\lambda^{\frac{p-2}{2}}\rho_2\right)^{p}} \,dz\\
        &\qquad\qquad \leq c\|u\|_{L^\infty(\Omega_T)}^{q-p}\lambda^{\frac{(2-p)(q-p)}{2}}\rho_2^{\alpha-q+p}\miint{Q_{\rho_2}^\lambda (z_0)}\frac{\Big|u-u_{Q_{\rho_2}^\lambda(z_0)}\Big|^{p}}{\left(\lambda^{\frac{p-2}{2}}\rho_2\right)^{p}} \,dz\\
        &\qquad\qquad \leq c\rho_2^{\alpha-\frac{(q-p)[(2-p)(n+2)+2\mu_2]}{2\mu_2}}\miint{Q_{\rho_2}^\lambda (z_0)}\frac{\Big|u-u_{Q_{\rho_2}^\lambda(z_0)}\Big|^{p}}{\left(\lambda^{\frac{p-2}{2}}\rho_2\right)^{p}} \,dz\\
        &\qquad\qquad \leq c\rho_2^{\alpha-\frac{4(q-p)}{p(n+2)-2n}}\miint{Q_{\rho_2}^\lambda (z_0)}\frac{\Big|u-u_{Q_{\rho_2}^\lambda(z_0)}\Big|^{p}}{\left(\lambda^{\frac{p-2}{2}}\rho_2\right)^{p}} \,dz\leq c\lambda^p
    \end{aligned}
    $$
    for some $c=c(\operatorname{data}_b)>1$. Next, by applying the method in \cite[Lemma 3.6]{Wontae2024}, we obtain
    $$
    \begin{aligned}
        &\miint{Q_{\rho_2}^\lambda (z_0)}\frac{\Big|u-u_{Q_{\rho_2}^\lambda(z_0)}\Big|^2}{\rho_2^2}\, dz\leq c\lambda^{p-1}S(u,Q_{\rho_2}^\lambda(z_0))^{\frac{1}{2}}
    \end{aligned}
    $$
    for some $c=c(\operatorname{data}_b)>1$. Finally, by \eqref{cond : p-phase condition}$_2$, we have 
    $$
    c\miint{Q_{\rho_2}^\lambda(z_0)} H(z,|F|)\, dz\leq c\lambda^p.
    $$ 
    Combining the above inequalities yields
    $$
    S(u,Q_{\rho_1}^\lambda(z_0))\leq \frac{c\rho_2^q}{(\rho_2-\rho_1)^q}\lambda^2+\frac{c\rho_2^2}{(\rho_2-\rho_1)^2}\lambda S(u,Q_{\rho_2}^\lambda (z_0))^\frac{1}{2}
    $$
    for some $c=c(\operatorname{data}_b)>1$. By Young's inequality, we get
    $$
    S(u,Q_{\rho_1}^\lambda(z_0))\leq \frac{1}{2}S(u,Q_{\rho_2}^\lambda (z_0))+c\left(\frac{\rho_2^q}{(\rho_2-\rho_1)^q}+\frac{\rho_2^4}{(\rho_2-\rho_1)^4}\right)\lambda^2.
    $$
    Therefore, the conclusion follows from Lemma \ref{lem : a standard iteration lemma}.
\end{proof}
Next, we estimate the first term on the right-hand side in Lemma \ref{lem : Caccioppoli inequality} under the assumptions \eqref{cond : main assumption with infty} and \eqref{cond : p-phase condition}.
\begin{lemma}\label{lem : estimate of the first term in Lemma 2.1 in p-intrinsic cylinder with bounded solution}
    There exist constants $c=c(\operatorname{data}_b)>1$ and $\theta_1=\theta_1(n)\in (0,1)$ such that for any $\theta\in (\theta_1,1)$,
    $$
    \begin{aligned}
        &\miint{Q_{2\rho}^\lambda (z_0)}\left(\frac{\Big|u-u_{Q_{2\rho}^\lambda (z_0)}\Big|^p}{\left(2\lambda^{\frac{p-2}{2}}\rho\right)^p}+a(z)\frac{\Big|u-u_{Q_{2\rho}^\lambda (z_0)}\Big|^q}{\left(2\lambda^{\frac{p-2}{2}}\rho\right)^q}\right)\, dz\\
        &\qquad\leq c \lambda^{(1-\theta)p} \miint{Q_{2\rho}^\lambda (z_0)}\left(\frac{\Big|u-u_{Q_{2\rho}^\lambda (z_0)}\Big|^{\theta p}}{\left(2\lambda^{\frac{p-2}{2}}\rho\right)^{\theta p}}+|Du|^{\theta p}\right)\, dz\\
        &\qquad\quad + c \lambda^{(1-\theta)p} \miint{Q_{2\rho}^\lambda (z_0)}\inf_{w\in Q_{2\rho}^\lambda (z_0)} a(w)^\theta \left(\frac{\Big|u-u_{Q_{2\rho}^\lambda (z_0)}\Big|^{\theta q}}{\left(2\lambda^{\frac{p-2}{2}}\rho\right)^{\theta q}}+|Du|^{\theta q}\right)\, dz.
    \end{aligned}
    $$
\end{lemma}
\begin{proof}
    Observe that
    \begin{align*}
        &\miint{Q_{2\rho}^\lambda (z_0)}\left(\frac{\Big|u-u_{Q_{2\rho}^\lambda (z_0)}\Big|^p}{\left(2\lambda^{\frac{p-2}{2}}\rho\right)^p}+a(z)\frac{\Big|u-u_{Q_{2\rho}^\lambda (z_0)}\Big|^q}{\left(2\lambda^{\frac{p-2}{2}}\rho\right)^q}\right)\, dz \\
        &\qquad\leq c \miint{Q_{2\rho}^\lambda (z_0)} \frac{\Big|u-u_{Q_{2\rho}^\lambda (z_0)}\Big|^p}{\left(2\lambda^{\frac{p-2}{2}}\rho\right)^p}\, dz + \miint{Q_{2\rho}^\lambda (z_0)} \inf_{w\in Q_{2\rho}^\lambda (z_0)} a(w) \frac{\Big|u-u_{Q_{2\rho}^\lambda (z_0)}\Big|^{q}}{\left(2\lambda^{\frac{p-2}{2}}\rho\right)^{ q}}\, dz \\
        &\qquad\qquad + [a]_\alpha (2\rho)^\alpha \miint{Q_{2\rho}^\lambda (z_0)} \frac{\Big|u-u_{Q_{2\rho}^\lambda (z_0)}\Big|^{q}}{\left(2\lambda^{\frac{p-2}{2}}\rho\right)^{q}}\, dz.
    \end{align*}
    To estimate the first and second terms on the right-hand side, we note from Lemma \ref{lem : Gagliardo-Nirenberg inequality} as in \cite[Lemma 3.7]{Wontae2024} and Lemma \ref{lem : estimate of S in p-intrinsic cylinder of bounded solution} that for $\theta \in \left(\frac{n}{n+2},1\right)$,
    $$
    \begin{aligned}
        &\miint{Q_{2\rho}^\lambda (z_0)} \frac{\Big|u-u_{Q_{2\rho}^\lambda (z_0)}\Big|^p}{\left(2\lambda^{\frac{p-2}{2}}\rho\right)^p}\, dz + \miint{Q_{2\rho}^\lambda (z_0)} \inf_{w\in Q_{2\rho}^\lambda (z_0)} a(w) \frac{\Big|u-u_{Q_{2\rho}^\lambda (z_0)}\Big|^{q}}{\left(2\lambda^{\frac{p-2}{2}}\rho\right)^{ q}}\, dz\\
        &\qquad\leq c \lambda^{(1-\theta)p} \miint{Q_{2\rho}^\lambda (z_0)}\left(\frac{\Big|u-u_{Q_{2\rho}^\lambda (z_0)}\Big|^{\theta p}}{\left(2\lambda^{\frac{p-2}{2}}\rho\right)^{\theta p}}+|Du|^{\theta p}\right)\, dz\\
        &\qquad\quad + c \lambda^{(1-\theta)p} \miint{Q_{2\rho}^\lambda (z_0)}\inf_{w\in Q_{2\rho}^\lambda (z_0)} a(w)^\theta \left(\frac{\Big|u-u_{Q_{2\rho}^\lambda (z_0)}\Big|^{\theta q}}{\left(2\lambda^{\frac{p-2}{2}}\rho\right)^{\theta q}}+|Du|^{\theta q}\right)\, dz
    \end{aligned}
    $$
    for some $c=c(\operatorname{data}_b)>1$.
    Then it follows from \eqref{cond : main assumption with infty}, \eqref{eq : decay estimate of rho and lambda} and \eqref{calculate mu_2} that
    $$
    \begin{aligned}
        &(2\rho)^\alpha \miint{Q_{2\rho}^\lambda (z_0)} \frac{\Big|u-u_{Q_{2\rho}^\lambda (z_0)}\Big|^{q}}{\left(2\lambda^{\frac{p-2}{2}}\rho\right)^{q}}\, dz\\
        &\qquad=(2\rho)^\alpha \miint{Q_{2\rho}^\lambda (z_0)} \frac{\Big|u-u_{Q_{2\rho}^\lambda (z_0)}\Big|^{q-p}}{\left(2\lambda^{\frac{p-2}{2}}\rho\right)^{q-p}}\frac{\Big|u-u_{Q_{2\rho}^\lambda (z_0)}\Big|^{p}}{\left(2\lambda^{\frac{p-2}{2}}\rho\right)^{p}}\, dz\\
        &\qquad\leq c\lambda^{\frac{(2-p)(q-p)}{2}}\rho^{\alpha-(q-p)}\miint{Q_{2\rho}^\lambda (z_0)}\frac{\Big|u-u_{Q_{2\rho}^\lambda (z_0)}\Big|^{p}}{\left(2\lambda^{\frac{p-2}{2}}\rho\right)^{p}}\, dz\\
        &\qquad\leq c\rho^{\alpha - \frac{(q-p)[(2-p)(n+2)+2\mu_2]}{2\mu_2}}\miint{Q_{2\rho}^\lambda (z_0)}\frac{\Big|u-u_{Q_{2\rho}^\lambda (z_0)}\Big|^{p}}{\left(2\lambda^{\frac{p-2}{2}}\rho\right)^{p}}\, dz\\
        &\qquad\leq c\rho^{\alpha - \frac{4(q-p)}{p(n+2)-2n}}\miint{Q_{2\rho}^\lambda (z_0)}\frac{\Big|u-u_{Q_{2\rho}^\lambda (z_0)}\Big|^{p}}{\left(2\lambda^{\frac{p-2}{2}}\rho\right)^{p}}\, dz\\
        &\qquad\leq c \lambda^{(1-\theta)p} \miint{Q_{2\rho}^\lambda (z_0)}\left(\frac{\Big|u-u_{Q_{2\rho}^\lambda (z_0)}\Big|^{\theta p}}{\left(2\lambda^{\frac{p-2}{2}}\rho\right)^{\theta p}}+|Du|^{\theta p}\right)\, dz\\
        &\qquad\quad + c \lambda^{(1-\theta)p} \miint{Q_{2\rho}^\lambda (z_0)}\inf_{w\in Q_{2\rho}^\lambda (z_0)} a(w)^\theta \left(\frac{\Big|u-u_{Q_{2\rho}^\lambda (z_0)}\Big|^{\theta q}}{\left(2\lambda^{\frac{p-2}{2}}\rho\right)^{\theta q}}+|Du|^{\theta q}\right)\, dz
    \end{aligned}
    $$
    for some $c=c(\operatorname{data}_b)>1$. Hence, we conclude that for any $\theta \in \left(\frac{n}{n+2},1\right)$,
    $$
    \begin{aligned}
        &\miint{Q_{2\rho}^\lambda (z_0)}\left(\frac{\Big|u-u_{Q_{2\rho}^\lambda (z_0)}\Big|^p}{\left(2\lambda^{\frac{p-2}{2}}\rho\right)^p}+a(z)\frac{\Big|u-u_{Q_{2\rho}^\lambda (z_0)}\Big|^q}{\left(2\lambda^{\frac{p-2}{2}}\rho\right)^q}\right)\, dz\\
        &\qquad \leq c \lambda^{(1-\theta)p} \miint{Q_{2\rho}^\lambda (z_0)}\left(\frac{\Big|u-u_{Q_{2\rho}^\lambda (z_0)}\Big|^{\theta p}}{\left(2\lambda^{\frac{p-2}{2}}\rho\right)^{\theta p}}+|Du|^{\theta p}\right)\, dz\\
        &\qquad\quad + c \lambda^{(1-\theta)p} \miint{Q_{2\rho}^\lambda (z_0)}\inf_{w\in Q_{2\rho}^\lambda (z_0)} a(w)^\theta \left(\frac{\Big|u-u_{Q_{2\rho}^\lambda (z_0)}\Big|^{\theta q}}{\left(2\lambda^{\frac{p-2}{2}}\rho\right)^{\theta q}}+|Du|^{\theta q}\right)\, dz
    \end{aligned}
    $$
    for some $c=c(\operatorname{data}_b)>1$.
\end{proof}
Now, we prove the reverse H\"{o}lder inequality in the $p$-intrinsic case.
\begin{lemma}\label{lem : the reverse Holder inequality in the p-intrinsic case of bounded solution}
    There exist constants $c=c(\operatorname{data}_b)>1$ and $\theta_0=\theta_0(n,p,q)\in (0,1)$ such that for any $\theta\in (\theta_0,1)$, 
    $$
    \miint{Q_\rho^\lambda (z_0)} H(z,|Du|)\, dz\leq c \left(\miint{Q_{2\rho}^\lambda (z_0)} [H(z,|Du|)]^\theta \, dz\right)^\frac{1}{\theta}+c\miint{Q_{2\rho}^\lambda (z_0)} H(z,|F|) \, dz.
    $$
\end{lemma}
\begin{proof}
    It follows from Lemma \ref{lem : Caccioppoli inequality} that 
    \begin{align}
        \miint{Q_\rho^\lambda (z_0)} H(z,|Du|)\, dz&\leq c \miint{Q_{2\rho}^\lambda (z_0)} \left(\frac{\Big|u-u_{Q_{2\rho}^\lambda (z_0)}\Big|^p}{\left(2\lambda^{\frac{p-2}{2}}\rho\right)^p}+a(z)\frac{\Big|u-u_{Q_{2\rho}^\lambda (z_0)}\Big|^q}{\left(2\lambda^{\frac{p-2}{2}}\rho\right)^q}\right)\, dz \nonumber\\\label{eq : estimation in H in p-intrinsic cylinder of bounded solution using Caccioppoli inequality}
        &\; +c\lambda^{p-2} \miint{Q_{2\rho}^\lambda (z_0)} \frac{\Big|u-u_{Q_{2\rho}^\lambda (z_0)}\Big|^2}{\left(2\lambda^{\frac{p-2}{2}}\rho\right)^2}\, dz+c\miint{Q_{2\rho}^\lambda (z_0)} H(z,|F|)\, dz,
    \end{align}
    where $c=c(n,p,q,\nu,L)>1$. Let $\theta_2\coloneq\max\left\{\theta_1,\frac{q-1}{p}\right\}$, where $\theta_1$ is defined in Lemma \ref{lem : estimate of the first term in Lemma 2.1 in p-intrinsic cylinder with bounded solution}. For $\theta \in (\theta_2,1)$, using Lemmas \ref{lem : estimate of the first term in Lemma 2.1 in p-intrinsic cylinder with bounded solution}, \ref{lem : p-intrinsic parabolic Poincare inequality of p-term in p-intrinsic cylinder with bounded solution}, \ref{lem : p-intrinsic parabolic Poincare inequality of q-term in p-intrinsic cylinder with bounded solution} and Young's inequality yields
    $$
    \begin{aligned}
        &\miint{Q_{2\rho}^\lambda (z_0)} \left(\frac{\Big|u-u_{Q_{2\rho}^\lambda (z_0)}\Big|^p}{\left(2\lambda^{\frac{p-2}{2}}\rho\right)^p}+a(z)\frac{\Big|u-u_{Q_{2\rho}^\lambda (z_0)}\Big|^q}{\left(2\lambda^{\frac{p-2}{2}}\rho\right)^q}\right)\, dz\\
        &\qquad \leq c \lambda^{(1-\theta)p} \miint{Q_{2\rho}^\lambda (z_0)} [H(z,|Du|)]^\theta \, dz\\
        &\qquad \quad + c \lambda^{\left(1-p+\frac{\alpha(p(n+2)-2n)}{8}\right)\theta p+p} \left(\miint{Q_{2\rho}^\lambda (z_0)} (|Du|+|F|)^{\theta p}\, dz\right)^{p-1-\frac{\alpha(p(n+2)-2n)}{8}}
    \end{aligned}
    $$
    for some $c=c(\operatorname{data}_b)>1$. Recall that $p-1-\frac{\alpha(p(n+2)-2n)}{8}>0$.
    Putting
    $$
    \beta \coloneq \min\left\{p-1-\frac{\alpha(p(n+2)-2n)}{8},\frac{1}{2}\right\},
    $$
    we have
    \begin{align}
        &\miint{Q_{2\rho}^\lambda (z_0)} \left(\frac{\Big|u-u_{Q_{2\rho}^\lambda (z_0)}\Big|^p}{\left(2\lambda^{\frac{p-2}{2}}\rho\right)^p}+a(z)\frac{\Big|u-u_{Q_{2\rho}^\lambda (z_0)}\Big|^q}{\left(2\lambda^{\frac{p-2}{2}}\rho\right)^q}\right)\, dz \nonumber\\
        &\qquad \leq c \lambda^{(1-\beta\theta)p} \left(\miint{Q_{2\rho}^\lambda (z_0)} [H(z,|Du|)]^\theta \, dz\right)^\beta\nonumber\\\label{eq : estimation of p,q term in H in p-intrinsic cylinder of bounded solution using Caccioppoli inequality}
        &\qquad \quad +c\lambda^{(1-\beta\theta)p}\left(\miint{Q_{2\rho}^\lambda (z_0)} H(z,|F|)\, dz\right)^{\beta\theta}.
    \end{align}
    On the other hand, we note that 
    $$
    \begin{aligned}
    -\frac{n}{2}\leq \frac{1}{2}\left(1-\frac{n}{\theta p}\right)-\left(1-\frac{1}{2}\right)\frac{n}{2}\quad \iff \quad \frac{2n}{(n+2)p}\leq \theta.
    \end{aligned}
    $$
    Since $\frac{2n}{n+2}< p \leq 2$, the assumption of Lemma \ref{lem : Gagliardo-Nirenberg inequality} with $p_1=2$, $p_2=\theta p$, $p_3=2$ and $\vartheta=\frac{1}{2}$ is satisfied. Hence we get from Lemmas \ref{lem : Gagliardo-Nirenberg inequality} and \ref{lem : estimate of S in p-intrinsic cylinder of bounded solution} that for $\theta\in \left(\frac{2n}{(n+2)p},1\right)$,
    $$
    \begin{aligned}
    &\miint{Q_{2\rho}^\lambda (z_0)} \frac{\Big|u-u_{Q_{2\rho}^\lambda (z_0)}\Big|^2}{\left(2\lambda^{\frac{p-2}{2}}\rho\right)^2}\, dz\\
    &\qquad \leq c\dashint_{I_{2\rho}(t_0)}\left(\dashint_{B_{2\rho}^\lambda(x_0)}\left(\frac{\Big|u-u_{Q_{2\rho}^\lambda (z_0)}\Big|^{\theta p}}{\left(2\lambda^{\frac{p-2}{2}}\rho\right)^{\theta p}}+|Du|^{\theta p}\right)\, dx\right)^\frac{1}{\theta p}\, dt\\
    &\qquad\quad \times (S(u,Q_{2\rho}^\lambda (z_0)))^\frac{1}{2}\\
    &\qquad \leq c \lambda \left(\miint{Q_{2\rho}^\lambda (z_0)} \left(\frac{\Big|u-u_{Q_{2\rho}^\lambda (z_0)}\Big|^{\theta p}}{\left(2\lambda^{\frac{p-2}{2}}\rho\right)^{\theta p}}+|Du|^{\theta p}\right)\,dz\right)^\frac{1}{\theta p}
    \end{aligned}
    $$
    for some $c=c(\operatorname{data}_b)>1$. By \eqref{cond : p-phase condition}$_2$ and Lemma \ref{lem : p-intrinsic parabolic Poincare inequality of p-term in p-intrinsic cylinder with bounded solution}, we have
    \begin{align} \label{eq : estimation of 2 term in H in p-intrinsic cylinder of bounded solution using Caccioppoli inequality}
        \lambda^{p-2} \miint{Q_{2\rho}^\lambda (z_0)} \frac{\Big|u-u_{Q_{2\rho}^\lambda (z_0)}\Big|^2}{\left(2\lambda^{\frac{p-2}{2}}\rho\right)^2}\, dz &\leq c\lambda^{p-\beta}\left(\miint{Q_{2\rho}^\lambda (z_0)} [H(z,|Du|)]^\theta\, dz\right)^{\frac{\beta}{\theta p}}\nonumber\\
        &\quad+ c\lambda^{p-\beta}\left(\miint{Q_{2\rho}^\lambda (z_0)} H(z,|F|)\, dz\right)^{\frac{\beta}{p}}.
    \end{align}
    Combining \eqref{eq : estimation in H in p-intrinsic cylinder of bounded solution using Caccioppoli inequality}, \eqref{eq : estimation of p,q term in H in p-intrinsic cylinder of bounded solution using Caccioppoli inequality} and \eqref{eq : estimation of 2 term in H in p-intrinsic cylinder of bounded solution using Caccioppoli inequality} implies that for $\theta\in (\theta_0, 1)$,
    $$
    \begin{aligned}
        \miint{Q_\rho^\lambda (z_0)} H(z,|Du|)\, dz&\leq c\lambda^{p-\beta}\left(\miint{Q_{2\rho}^\lambda (z_0)} [H(z,|Du|)]^\theta\, dz\right)^{\frac{\beta}{\theta p}}\\
        &\quad+ c\lambda^{p-\beta}\left(\miint{Q_{2\rho}^\lambda (z_0)} H(z,|F|)\, dz\right)^{\frac{\beta}{p}},
    \end{aligned}
    $$
    where $\theta_0=\max\left\{\theta_2,\frac{2n}{(n+2)p}\right\}$ and $c=c(\operatorname{data}_b)>1$. It follows from Young's inequality that
    $$
    \begin{aligned}
        &\miint{Q_\rho^\lambda (z_0)} H(z,|Du|)\, dz\\
        &\qquad\leq \frac{1}{2}\lambda^p+c\left(\miint{Q_{2\rho}^\lambda (z_0)} [H(z,|Du|)]^\theta\, dz\right)^{\frac{1}{\theta}}+c\miint{Q_{2\rho}^\lambda (z_0)} H(z,|F|)\, dz.
    \end{aligned}
    $$
    Thus, we conclude from \eqref{cond : p-phase condition}$_3$ that
    $$
    \miint{Q_\rho^\lambda (z_0)} H(z,|Du|)\, dz\leq c\left(\miint{Q_{2\rho}^\lambda (z_0)} [H(z,|Du|)]^\theta\, dz\right)^{\frac{1}{\theta}}+c\miint{Q_{2\rho}^\lambda (z_0)} H(z,|F|)\, dz.
    $$
\end{proof}
\subsubsection{Assumption \eqref{cond : main assumption with s}}
From now on, we assume \eqref{cond : main assumption with s} instead of \eqref{cond : main assumption with infty}. First, we establish a $p$-intrinsic parabolic Poincar\'{e} inequality.
    \begin{lemma}\label{lem : p-intrinsic parabolic Poincare inequality of p-term in p-intrinsic cylinder with s}
    For $\sigma\in[2\rho,4\rho]$ and $\theta \in \left(\frac{q-1}{p},1\right]$, there exists a constant $c=c(\operatorname{data}_s)$ $>1$ such that
    \begin{align*}
        &\miint{Q_\sigma^\lambda(z_0)}\frac{\Big|u-u_{Q_\sigma^\lambda(z_0)}\Big|^{\theta p}}{(\lambda^{\frac{p-2}{2}}\sigma)^{\theta p}}\, dz\\
        &\qquad\qquad\leq c\miint{Q_\sigma^\lambda(z_0)} [H(z,|Du|)]^\theta\, dz\\
        &\qquad\qquad\quad +c\lambda^{\left(2-p+\frac{\alpha \mu_s}{n+s}\right)\theta p}\left(\miint{Q_\sigma^\lambda(z_0)} (|Du|+|F|)^{\theta p}\, dz\right)^{p-1-\frac{\alpha\mu_s}{n+s}}.
    \end{align*}
\end{lemma}
\begin{proof}
    By Lemmas \ref{lem : semi-Parabolic Poincare inequality} and \ref{lem : last term estimate in Lemma lem : semi-Parabolic Poincare inequality whenever p-intrinsic cylinder}, there exists a constant $c=c(\operatorname{data}_s)>1$ such that 
    \begin{align*}
        \miint{Q_\sigma^\lambda(z_0)}\frac{\Big|u-u_{Q_\sigma^\lambda(z_0)}\Big|^{\theta p}}{(\lambda^{\frac{p-2}{2}}\sigma)^{\theta p}}\, dz &\leq  c \miint{Q_\sigma^\lambda(z_0)}|D u|^{\theta p} \,dz\\
        &\quad +c\left(\lambda^{2-p}\miint{Q_\sigma^\lambda(z_0)} (|Du|+|F|)^{p-1}\, dz\right)^{\theta p}\\
        &\quad +c\left(\lambda^{1-p+\frac{p}{q}}\miint{Q_\sigma^\lambda(z_0)} a(z)^{\frac{q-1}{q}}(|Du|+|F|)^{q-1}\, dz\right)^{\theta p}.
    \end{align*}
    Note that $\frac{s}{n+s}<1$ implies  
    \begin{align*}
        p-1-\frac{\alpha\mu_s}{n+s}\geq p-1 -\frac{\mu_s}{n+s} > \frac{(2-p)(n-2)}{4}\geq 0.
    \end{align*}
    Using \eqref{cond : p-phase condition} and H\"{o}lder's inequality gives
    $$
    \begin{aligned}
        &\lambda^{(2-p)\theta p}\left(\miint{Q_\sigma^\lambda(z_0)} (|Du|+|F|)^{p-1}\, dz\right)^{\theta p}\\
        &\;  \leq \lambda^{(2-p)\theta p}\left(\miint{Q_\sigma^\lambda(z_0)} (|Du|+|F|)^{\theta p}\, dz\right)^{p-1}\\
        &\; =  \lambda^{(2-p)\theta p} \left(\miint{Q_\sigma^\lambda(z_0)} (|Du|+|F|)^{\theta p}\, dz\right)^{\frac{\alpha\mu_s}{n+s}} \left(\miint{Q_\sigma^\lambda(z_0)} (|Du|+|F|)^{\theta p}\, dz\right)^{p-1-\frac{\alpha\mu_s}{n+s}}\\
        &\; \leq \lambda^{\left(2-p+\frac{\alpha\mu_s}{n+s}\right)\theta p}\left(\miint{Q_\sigma^\lambda(z_0)} (|Du|+|F|)^{\theta p}\, dz\right)^{p-1-\frac{\alpha\mu_s}{n+s}}.
    \end{aligned}
    $$
    Moreover, it follows from \eqref{cond : p-phase condition} and H\"{o}lder's inequality that
    $$
    \begin{aligned}
        &\left(\lambda^{1-p+\frac{p}{q}}\miint{Q_\sigma^\lambda(z_0)} a(z)^{\frac{q-1}{q}}(|Du|+|F|)^{q-1}\, dz\right)^{\theta p}\\
        &\qquad \leq c\lambda^{\left(1-p+\frac{p}{q}+\frac{(p-q)(q-1)}{q}\right)\theta p}\left(\miint{Q_\sigma^\lambda(z_0)} (|Du|+|F|)^{q-1}\, dz\right)^{\theta p}\\
        &\qquad \leq c\lambda^{\left(2-p+\frac{\alpha\mu_s}{n+s}\right)\theta p}\left(\miint{Q_\sigma^\lambda(z_0)} (|Du|+|F|)^{\theta p}\, dz\right)^{p-1-\frac{\alpha\mu_s}{n+s}}
    \end{aligned}
    $$
    for some $c=c(\operatorname{data}_s)>1$. This completes the proof.
\end{proof}
\begin{lemma}\label{lem : p-intrinsic parabolic Poincare inequality of q-term in p-intrinsic cylinder with s}
    For $\sigma\in[2\rho,4\rho]$ and $\theta \in \left(\frac{q-1}{p},1\right]$, there exists a constant $c=c(\operatorname{data}_s)$ $>1$ such that
    \begin{align*}
        &\miint{Q_\sigma^\lambda(z_0)}\inf_{w\in Q_\sigma^\lambda (z_0)} a(w)^\theta \frac{\Big|u-u_{Q_\sigma^\lambda(z_0)}\Big|^{\theta q}}{(\lambda^{\frac{p-2}{2}}\sigma)^{\theta q}}\, dz\\
        &\qquad\qquad\leq c\miint{Q_\sigma^\lambda(z_0)} [H(z,|Du|)]^\theta\, dz\\
        &\qquad\qquad\quad +c\lambda^{\left(2-p+\frac{\alpha\mu_s}{n+s}\right)\theta p}\left(\miint{Q_\sigma^\lambda(z_0)} (|Du|+|F|)^{\theta p}\, dz\right)^{p-1-\frac{\alpha\mu_s}{n+s}}.
    \end{align*}
\end{lemma}
\begin{proof}
    By Lemmas \ref{lem : semi-Parabolic Poincare inequality} and \ref{lem : last term estimate in Lemma lem : semi-Parabolic Poincare inequality whenever p-intrinsic cylinder}, there exists a constant $c=c(\operatorname{data}_s)>1$ such that 
    \begin{align*}
        &\miint{Q_\sigma^\lambda(z_0)}\inf_{w\in Q_\sigma^\lambda (z_0)} a(w)^\theta \frac{\Big|u-u_{Q_\sigma^\lambda(z_0)}\Big|^{\theta q}}{(\lambda^{\frac{p-2}{2}}\sigma)^{\theta q}}\, dz\\
        &\qquad\qquad\leq  c \miint{Q_\sigma^\lambda(z_0)} \inf_{w\in Q_\sigma^\lambda (z_0)} a(w)^\theta|D u|^{\theta q} \,dz\\
        &\qquad\qquad\quad +c\inf_{w\in Q_\sigma^\lambda (z_0)} a(w)^\theta\left(\lambda^{2-p}\miint{Q_\sigma^\lambda(z_0)} (|Du|+|F|)^{p-1}\, dz\right)^{\theta q}\\
        &\qquad\qquad\quad +c\inf_{w\in Q_\sigma^\lambda (z_0)} a(w)^\theta\left(\lambda^{1-p+\frac{p}{q}}\miint{Q_\sigma^\lambda(z_0)} a(z)^{\frac{q-1}{q}}(|Du|+|F|)^{q-1}\, dz\right)^{\theta q}.
    \end{align*}
    By \eqref{cond : p-phase condition} and H\"{o}lder's inequality, we have
    $$
    \begin{aligned}
        &\inf_{w\in Q_\sigma^\lambda (z_0)} a(w)^\theta\left(\lambda^{2-p}\miint{Q_\sigma^\lambda(z_0)} (|Du|+|F|)^{p-1}\, dz\right)^{\theta q}\\
        &\qquad \leq \inf_{w\in Q_\sigma^\lambda (z_0)} a(w)^\theta\lambda^{(2-p)\theta q}\left(\miint{Q_\sigma^\lambda(z_0)} (|Du|+|F|)^{\theta p}\, dz\right)^{\frac{q(p-1)}{p}}\\
        &\qquad \leq  c\lambda^{\theta p\left[1-\frac{q}{p}+(2-p)\frac{q}{p}+(p-1)(\frac{q}{p}-1)+\frac{\alpha\mu_s}{n+s}\right]}\left(\miint{Q_\sigma^\lambda(z_0)} (|Du|+|F|)^{\theta p}\, dz\right)^{p-1-\frac{\alpha\mu_s}{n+s}}\\
        &\qquad= c\lambda^{\theta p \left(2-p+\frac{\alpha\mu_s}{n+s}\right)} \left(\miint{Q_\sigma^\lambda(z_0)} (|Du|+|F|)^{\theta p}\, dz\right)^{p-1-\frac{\alpha\mu_s}{n+s}}
    \end{aligned}
    $$
    for some $c=c(\operatorname{data}_s)> 1$. Also, arguing in the same way, we get the following:
    $$
    \begin{aligned}
        &\inf_{w\in Q_\sigma^\lambda (z_0)} a(w)^\theta\left(\lambda^{1-p+\frac{p}{q}}\miint{Q_\sigma^\lambda(z_0)} a(z)^{\frac{q-1}{q}}(|Du|+|F|)^{q-1}\, dz\right)^{\theta q}\\
        &\qquad\qquad \leq c\lambda^{\theta p \left(2-p+\frac{\alpha\mu_s}{n+s}\right)} \left(\miint{Q_\sigma^\lambda(z_0)} (|Du|+|F|)^{\theta p}\, dz\right)^{p-1-\frac{\alpha\mu_s}{n+s}}
    \end{aligned}
    $$
    for some $c=c(\operatorname{data}_s)> 1$.
\end{proof}
Now, we proceed to estimate 
$$
S(u,Q_{2\rho}^\lambda(z_0))=\sup_{I_{2\rho}(t_0)}\dashint_{B_{2\rho}^\lambda(x_0)}\frac{\Big|u-u_{Q_{2\rho}^\lambda (z_0)}\Big|^2}{\left(2\lambda^{\frac{p-2}{2}}\rho\right)^2}\, dx.
$$
\begin{lemma}\label{lem : estimate of S in p-intrinsic cylinder of s}
    There exists a constant $c=c(\operatorname{data}_s)>1$ such that
    $$
    S(u,Q_{2\rho}^\lambda (z_0))=\sup_{I_{2\rho}(t_0)}\dashint_{B_{2\rho}^\lambda(x_0)}\frac{\Big|u-u_{Q_{2\rho}^\lambda (z_0)}\Big|^2}{\left(2\lambda^{\frac{p-2}{2}}\rho\right)^2}\, dx \leq c\lambda^2.
    $$
\end{lemma}
\begin{proof}
    Let $2\rho\leq \rho_1<\rho_2\leq 4\rho$. By Lemma \ref{lem : Caccioppoli inequality}, there exists a constant $c=c(n,p,q,\nu,L)>1$ such that
    \begin{align*}
        &\lambda^{p-2}S(u,Q_{\rho_1}^\lambda (z_0)) \\
        &\qquad\leq \frac{c\rho_2^q}{(\rho_2-\rho_1)^q}\miint{Q_{\rho_2}^\lambda (z_0)}\left(\frac{\Big|u-u_{Q_{\rho_2}^\lambda (z_0)}\Big|^p}{\left(\lambda^{\frac{p-2}{2}}\rho_2\right)^p}+a(z)\frac{\Big|u-u_{Q_{\rho_2}^\lambda(z_0)}\Big|^q}{\left(\lambda^{\frac{p-2}{2}}\rho_2\right)^q}\right)\, dz \\
        &\qquad\quad+\frac{c\rho_2^2}{(\rho_2-\rho_1)^2}\miint{Q_{\rho_2}^\lambda (z_0)}\frac{\Big|u-u_{Q_{\rho_2}^\lambda(z_0)}\Big|^2}{\rho_2^2}\, dz+c\miint{Q_{\rho_2}^\lambda (z_0)} H(z,|F|)\, dz.
    \end{align*}
    By \eqref{cond : p-phase condition}$_2$ and Lemma \ref{lem : p-intrinsic parabolic Poincare inequality of p-term in p-intrinsic cylinder with s}, we obtain
    \begin{equation}    \label{eq : p-term in S in p-intrinsic cylinder}
    \miint{Q_{\rho_2}^\lambda (z_0)}\frac{\Big|u-u_{Q_{\rho_2}^\lambda (z_0)}\Big|^p}{\left(\lambda^{\frac{p-2}{2}}\rho_2\right)^p}\, dz \leq c\lambda^p
    \end{equation}
    for some $c=c(\operatorname{data}_s)>1$. Note that
    $$
    \begin{aligned}
        \miint{Q_{\rho_2}^\lambda (z_0)}a(z)\frac{\Big|u-u_{Q_{\rho_2}^\lambda(z_0)}\Big|^q}{\left(\lambda^{\frac{p-2}{2}}\rho_2\right)^q}\, dz &\leq \miint{Q_{\rho_2}^\lambda (z_0)} \inf_{w\in Q_{\rho_2}^\lambda(z_0)} a(w)\frac{\Big|u-u_{Q_{\rho_2}^\lambda(z_0)}\Big|^q}{\left(\lambda^{\frac{p-2}{2}}\rho_2\right)^q}\, dz\\
        &\qquad +[a]_\alpha \rho_2^\alpha\miint{Q_{\rho_2}^\lambda (z_0)}\frac{\Big|u-u_{Q_{\rho_2}^\lambda(z_0)}\Big|^q}{\left(\lambda^{\frac{p-2}{2}}\rho_2\right)^q}\, dz.
    \end{aligned}
    $$
    We deduce from \eqref{cond : p-phase condition}$_2$ and Lemma \ref{lem : p-intrinsic parabolic Poincare inequality of q-term in p-intrinsic cylinder with s} that
    $$
    \miint{Q_{\rho_2}^\lambda (z_0)} \inf_{w\in Q_{\rho_2}^\lambda(z_0)} a(w)\frac{\Big|u-u_{Q_{\rho_2}^\lambda(z_0)}\Big|^q}{\left(\lambda^{\frac{p-2}{2}}\rho_2\right)^q}\, dz\leq c\lambda^p
    $$
    for some $c=c(\operatorname{data}_s)>1$. Note that
    $$
    \begin{aligned}
        -\frac{n}{q}\leq \frac{p}{q}\left(1-\frac{n}{p}\right)-\left(1-\frac{p}{q}\right)\frac{n}{s}\quad \iff\quad q\leq p+\frac{ps}{n}.
    \end{aligned}
    $$
    Since $n< 2(n+s)$, $\alpha\leq 1$ and $p-2 \leq 0$, we obtain from \eqref{cond : main assumption with s} that
    $$
    q\leq p +\frac{\alpha \mu_s}{n+s}\leq p+\frac{(n(p-2)+2p)s}{4(n+s)}\leq p+\frac{ps}{n},
    $$
    which implies that the assumption of Lemma \ref{lem : Gagliardo-Nirenberg inequality} with $p_1=q$, $p_2=p$, $p_3=s$ and $\vartheta=\frac{p}{q}$ is satisfied. Thus, it follows from \eqref{cond : main assumption with s}, \eqref{eq : p-term in S in p-intrinsic cylinder}, \eqref{cond : p-phase condition}$_2$ and Lemma \ref{lem : Gagliardo-Nirenberg inequality} that
    $$
    \begin{aligned}
        \rho_2^\alpha\miint{Q_{\rho_2}^\lambda (z_0)}\frac{\Big|u-u_{Q_{\rho_2}^\lambda(z_0)}\Big|^q}{\left(\lambda^{\frac{p-2}{2}}\rho_2\right)^q}\, dz &\leq c\rho_2^\alpha \left(\miint{Q_{\rho_2}^\lambda (z_0)}\left(\frac{\Big|u-u_{Q_{\rho_2}^\lambda(z_0)}\Big|^p}{\left(\lambda^{\frac{p-2}{2}}\rho_2\right)^p}+|Du|^p\right)\, dz\right)\\
        &\qquad\times \left(\sup_{I_{\rho_2}(t_0)}\dashint_{B_{\rho_2}^\lambda (x_0)} \frac{\Big|u-u_{Q_{\rho_2}^\lambda(z_0)}\Big|^s}{\left(\lambda^{\frac{p-2}{2}}\rho_2\right)^s}\, dx\right)^{\frac{q-p}{s}}\\
        &\leq c \rho_2^\alpha \lambda^p\left(\sup_{I_{\rho_2}(t_0)}\dashint_{B_{\rho_2}^\lambda (x_0)} \frac{|u|^s}{\left(\lambda^{\frac{p-2}{2}}\rho_2\right)^s}\, dx\right)^{\frac{q-p}{s}}\\
        &\leq c\rho_2^{\alpha-\frac{(q-p)(n+s)}{s}}\lambda^{\frac{(2-p)(q-p)(n+s)}{2s}}\lambda^p.
    \end{aligned}
    $$
    Since $2\rho\leq \rho_2\leq 4\rho$ and $\mu_s=\frac{s\mu_2}{2}$, we observe from \eqref{eq : decay estimate of rho and lambda} and \eqref{calculate mu_2} that
    $$
    \begin{aligned}
        \rho_2^{\alpha-\frac{(q-p)(n+s)}{s}}\lambda^{\frac{(2-p)(q-p)(n+s)}{2s}} \leq c \rho_2^{\alpha -\frac{(q-p)(n+s)(2\mu_2+(2-p)(n+2))}{2s\mu_2}}=c\rho_2^{\alpha - \frac{(q-p)(n+s)}{\mu_s}}.
    \end{aligned}
    $$
    Since $q\leq p+ \frac{\alpha \mu_s}{n+s}$ implies $\alpha -\frac{(q-p)(n+s)}{\mu_s}\geq 0$, we get
    $$
    \begin{aligned}
        \rho_2^\alpha\miint{Q_{\rho_2}^\lambda (z_0)}\frac{\Big|u-u_{Q_{\rho_2}^\lambda(z_0)}\Big|^q}{\left(\lambda^{\frac{p-2}{2}}\rho_2\right)^q}\, dz \leq c \lambda^p
    \end{aligned}
    $$
    for some $c=c(\operatorname{data}_s)>1$. Next, in the same manner as in \cite[Lemma 3.6]{Wontae2024}, we obtain
    $$
    \begin{aligned}
        &\miint{Q_{\rho_2}^\lambda (z_0)}\frac{\Big|u-u_{Q_{\rho_2}^\lambda(z_0)}\Big|^2}{\rho_2^2}\, dz\leq c\lambda^{p-1}S(u,Q_{\rho_2}^\lambda(z_0))^{\frac{1}{2}}
    \end{aligned}
    $$
    for some $c=c(\operatorname{data}_s)>1$. Finally, we obtain from \eqref{cond : p-phase condition}$_2$ that 
    $$
    c\miint{Q_{\rho_2}^\lambda (z_0)} H(z,|F|)\, dz\leq c\lambda^p.
    $$
    Combining the above inequalities gives
    $$
    S(u,Q_{\rho_1}^\lambda(z_0))\leq \frac{c\rho_2^q}{(\rho_2-\rho_1)^q}\lambda^2+\frac{c\rho_2^2}{(\rho_2-\rho_1)^2}\lambda S(u,Q_{\rho_2}^\lambda (z_0))^\frac{1}{2}
    $$
    for some $c=c(\operatorname{data}_s)>1$. By Young's inequality, we have
    $$
    S(u,Q_{\rho_1}^\lambda(z_0))\leq \frac{1}{2}S(u,Q_{\rho_2}^\lambda (z_0))+c\left(\frac{\rho_2^q}{(\rho_2-\rho_1)^q}+\frac{\rho_2^4}{(\rho_2-\rho_1)^4}\right)\lambda^2.
    $$
    Therefore, the conclusion follows from Lemma \ref{lem : a standard iteration lemma}.
\end{proof}
We now estimate the first term on the right-hand side in Lemma \ref{lem : Caccioppoli inequality}, assuming \eqref{cond : main assumption with s} and \eqref{cond : p-phase condition}.
\begin{lemma}\label{lem : estimate of the first term in Lemma 2.1 in p-intrinsic cylinder with s}
    There exist constants $c=c(\operatorname{data}_s)>1$ and $\theta_1=\theta_1(n,p,q,s)\in (0,1)$ such that for any $\theta\in (\theta_1,1)$, we obtain
    $$
    \begin{aligned}
        &\miint{Q_{2\rho}^\lambda (z_0)}\left(\frac{\Big|u-u_{Q_{2\rho}^\lambda (z_0)}\Big|^p}{\left(2\lambda^{\frac{p-2}{2}}\rho\right)^p}+a(z)\frac{\Big|u-u_{Q_{2\rho}^\lambda (z_0)}\Big|^q}{\left(2\lambda^{\frac{p-2}{2}}\rho\right)^q}\right)\, dz\\
        &\qquad\leq c \lambda^{(1-\theta)p} \miint{Q_{2\rho}^\lambda (z_0)}\left(\frac{\Big|u-u_{Q_{2\rho}^\lambda (z_0)}\Big|^{\theta p}}{\left(2\lambda^{\frac{p-2}{2}}\rho\right)^{\theta p}}+|Du|^{\theta p}\right)\, dz\\
        &\qquad\quad + c \lambda^{(1-\theta)p} \miint{Q_{2\rho}^\lambda (z_0)}\inf_{w\in Q_{2\rho}^\lambda (z_0)} a(w)^\theta \left(\frac{\Big|u-u_{Q_{2\rho}^\lambda (z_0)}\Big|^{\theta q}}{\left(2\lambda^{\frac{p-2}{2}}\rho\right)^{\theta q}}+|Du|^{\theta q}\right)\, dz.
    \end{aligned}
    $$
\end{lemma}
\begin{proof}
    Observe that
    \begin{align*}
        &\miint{Q_{2\rho}^\lambda (z_0)}\left(\frac{\Big|u-u_{Q_{2\rho}^\lambda (z_0)}\Big|^p}{\left(2\lambda^{\frac{p-2}{2}}\rho\right)^p}+a(z)\frac{\Big|u-u_{Q_{2\rho}^\lambda (z_0)}\Big|^q}{\left(2\lambda^{\frac{p-2}{2}}\rho\right)^q}\right)\, dz \\
        &\qquad\leq c \miint{Q_{2\rho}^\lambda (z_0)} \frac{\Big|u-u_{Q_{2\rho}^\lambda (z_0)}\Big|^p}{\left(2\lambda^{\frac{p-2}{2}}\rho\right)^p}\, dz + \miint{Q_{2\rho}^\lambda (z_0)} \inf_{w\in Q_{2\rho}^\lambda (z_0)} a(w) \frac{\Big|u-u_{Q_{2\rho}^\lambda (z_0)}\Big|^{q}}{\left(2\lambda^{\frac{p-2}{2}}\rho\right)^{ q}}\, dz \\
        &\qquad\qquad + [a]_\alpha (2\rho)^\alpha \miint{Q_{2\rho}^\lambda (z_0)} \frac{\Big|u-u_{Q_{2\rho}^\lambda (z_0)}\Big|^{q}}{\left(2\lambda^{\frac{p-2}{2}}\rho\right)^{q}}\, dz.
    \end{align*}
    To estimate the first and second terms on the right-hand side, we deduce from Lemma \ref{lem : Gagliardo-Nirenberg inequality}, similarly to \cite[Lemma 3.7]{Wontae2024}, and Lemma \ref{lem : estimate of S in p-intrinsic cylinder of s} that for $\theta \in \left(\frac{n}{n+2},1\right)$,
    $$
    \begin{aligned}
        &\miint{Q_{2\rho}^\lambda (z_0)} \frac{\Big|u-u_{Q_{2\rho}^\lambda (z_0)}\Big|^p}{\left(2\lambda^{\frac{p-2}{2}}\rho\right)^p}\, dz + \miint{Q_{2\rho}^\lambda (z_0)} \inf_{w\in Q_{2\rho}^\lambda (z_0)} a(w) \frac{\Big|u-u_{Q_{2\rho}^\lambda (z_0)}\Big|^{q}}{\left(2\lambda^{\frac{p-2}{2}}\rho\right)^{ q}}\, dz\\
        &\qquad\leq c \lambda^{(1-\theta)p} \miint{Q_{2\rho}^\lambda (z_0)}\left(\frac{\Big|u-u_{Q_{2\rho}^\lambda (z_0)}\Big|^{\theta p}}{\left(2\lambda^{\frac{p-2}{2}}\rho\right)^{\theta p}}+|Du|^{\theta p}\right)\, dz\\
        &\qquad\quad + c \lambda^{(1-\theta)p} \miint{Q_{2\rho}^\lambda (z_0)}\inf_{w\in Q_{2\rho}^\lambda (z_0)} a(w)^\theta \left(\frac{\Big|u-u_{Q_{2\rho}^\lambda (z_0)}\Big|^{\theta q}}{\left(2\lambda^{\frac{p-2}{2}}\rho\right)^{\theta q}}+|Du|^{\theta q}\right)\, dz
    \end{aligned}
    $$
    for some $c=c(\operatorname{data}_s)>1$. On the other hand, to estimate the last term, we treat the cases $2\leq s \leq  4$ and $4<s<\infty$ separately. First, we assume $2\leq s\leq 4$. To use Lemma \ref{lem : Gagliardo-Nirenberg inequality} with $p_1=q$, $p_2=\theta p$, $p_3=2$ and $\vartheta=\frac{\theta p}{q}$ for any $\theta\in \left(\frac{nq}{p(n+2)},1\right)$, we check that $\frac{nq}{p(n+2)}<1$ and the assumption in Lemma \ref{lem : Gagliardo-Nirenberg inequality} is satisfied. Since $\mu_s=\frac{(p(n+2)-2n)s}{4}$ and $\alpha\leq 1$, \eqref{cond : main assumption with s} implies
    $$
    \begin{aligned}
        \frac{nq}{p(n+2)}\leq \frac{n}{n+2}\left(1+\frac{\alpha \mu_s}{p(n+s)}\right)&\leq \frac{n}{n+2}\left(1+\frac{(p(n+2)-2n)s}{4p(n+s)}\right)\\
        &=\frac{(4p+ps-2s)n^2+6psn}{4pn^2+(4ps+8p)n+8ps}.
    \end{aligned}
    $$
    Since $4pn^2+(4ps+8p)n+8ps-((4p+ps-2s)n^2+6psn)=s(2-p)n^2 + 2p(4-s)n + 8ps>0$,
    $$
    \frac{nq}{p(n+2)}\leq \frac{(4p+ps-2s)n^2+6psn}{4pn^2+(4ps+8p)n+8ps} <1.
    $$
    Next, we note that
    $$
    \begin{aligned}
        -\frac{n}{q}\leq \frac{\theta p}{q}\left(1-\frac{n}{\theta p}\right)-\left(1-\frac{\theta p}{q}\right)\frac{n}{2}\quad \iff \quad \frac{n q}{p(n+2)}\leq \theta,
    \end{aligned}
    $$
    and so, the assumption of Lemma \ref{lem : Gagliardo-Nirenberg inequality} holds for $\theta \in \left(\frac{nq}{p(n+2)},1\right)$. Thus, we obtain from Lemmas \ref{lem : Gagliardo-Nirenberg inequality} and \ref{lem : estimate of S in p-intrinsic cylinder of s} that
    \begin{align*}
    &(2\rho)^\alpha \miint{Q_{2\rho}^\lambda (z_0)} \frac{\Big|u-u_{Q_{2\rho}^\lambda (z_0)}\Big|^{q}}{\left(2\lambda^{\frac{p-2}{2}}\rho\right)^{q}}\, dz\\
    &\qquad\leq c \miint{Q_{2\rho}^\lambda (z_0)} \left(\frac{\Big|u-u_{Q_{2\rho}^\lambda (z_0)}\Big|^{\theta p}}{\left(2\lambda^{\frac{p-2}{2}}\rho\right)^{\theta p} } + |Du|^{\theta p}\right)\, dz\\
    &\qquad\quad \times \sup_{I_{2\rho}(t_0)}\left(\dashint_{B_{2\rho}^\lambda (x_0)} \frac{\Big|u-u_{Q_{2\rho}^\lambda (z_0)}\Big|^{2}}{\left(2\lambda^{\frac{p-2}{2}}\rho\right)^{2}}\, dx\right)^{\frac{(1-\theta)p}{2}}\\
    &\qquad\quad \times (2\rho)^\alpha\sup_{I_{2\rho}(t_0)}\left(\dashint_{B_{2\rho}^\lambda (x_0)} \frac{\Big|u-u_{Q_{2\rho}^\lambda (z_0)}\Big|^{2}}{\left(2\lambda^{\frac{p-2}{2}}\rho\right)^{2}}\, dx\right)^{\frac{q-p}{2}}\\
    &\qquad\leq c \lambda^{(1-\theta)p}\miint{Q_{2\rho}^\lambda (z_0)} \left(\frac{\Big|u-u_{Q_{2\rho}^\lambda (z_0)}\Big|^{\theta p}}{\left(2\lambda^{\frac{p-2}{2}}\rho\right)^{\theta p} } + |Du|^{\theta p}\right)\, dz\\
    &\qquad\quad \times (2\rho)^\alpha\sup_{I_{2\rho}(t_0)}\left(\dashint_{B_{2\rho}^\lambda (x_0)} \frac{\Big|u-u_{Q_{2\rho}^\lambda (z_0)}\Big|^{s}}{\left(2\lambda^{\frac{p-2}{2}}\rho\right)^{s}}\, dx\right)^{\frac{q-p}{s}}.
    \end{align*}
    As in the proof of Lemma \ref{lem : estimate of S in p-intrinsic cylinder of s}, we obtain
    $$(2\rho)^\alpha\sup_{I_{2\rho}(t_0)}\left(\dashint_{B_{2\rho}^\lambda (x_0)} \frac{\Big|u-u_{Q_{2\rho}^\lambda (z_0)}\Big|^{s}}{\left(2\lambda^{\frac{p-2}{2}}\rho\right)^{s}}\, dx\right)^{\frac{q-p}{s}}\leq c(\operatorname{data}_s),
    $$
    and hence
    \begin{equation*}
    \begin{aligned}
        &(2\rho)^\alpha \miint{Q_{2\rho}^\lambda (z_0)} \frac{\Big|u-u_{Q_{2\rho}^\lambda (z_0)}\Big|^{q}}{\left(2\lambda^{\frac{p-2}{2}}\rho\right)^{q}}\, dz\\
        &\qquad\qquad\qquad\qquad\leq c\lambda^{(1-\theta)p}\miint{Q_{2\rho}^\lambda (z_0)} \left(\frac{\Big|u-u_{Q_{2\rho}^\lambda (z_0)}\Big|^{\theta p}}{\left(2\lambda^{\frac{p-2}{2}}\rho\right)^{\theta p} } + |Du|^{\theta p}\right)\, dz.
    \end{aligned}
    \end{equation*}
    Next, we assume that $4<s<\infty$. Let
    $$
    \theta \in \left(\frac{ps(s-3)-2(q-p)}{ps(s-3)},1\right)\quad\text{and}\quad \tilde{p}=\frac{2s(q-p\theta)}{ps(1-\theta)+2(q-p)}.
    $$
    Since $q-p<1$ and $s>4> 2$, we get $\theta >0$ and $\tilde{p}<s$. Also, by the range of $\theta$, we obtain $s-1<\tilde{p}$. Since $2<s-1<\tilde{p}$ and $\frac{q}{p}\leq 1 +\frac{\mu_s}{(n+s)p}$, we have
    $$
    \begin{aligned}
        \frac{nq}{p(n+\tilde{p})}< \frac{nq}{(n+s-1)p}&< \frac{n}{n+s-1}\left(1+\frac{(p(n+2)-2n)s}{4p(n+s)}\right)\\
        &=\frac{(4p+ps-2s)n^2+6psn}{4pn^2+(8ps-4p)n+4ps(s-1)}.
    \end{aligned}
    $$
    Since $4pn^2+(8ps-4p)n+4ps(s-1)-((4p+ps-2s)n^2+6psn)=s(2-p)n^2 + 2p(s-2)n +4ps(s-1)>0$, we see that
    $$
    \begin{aligned}
        \frac{nq}{p(n+\tilde{p})}<\frac{(4p+ps-2s)n^2+6psn}{4pn^2+(8ps-4p)n+4ps(s-1)}<1.
    \end{aligned}
    $$
    Since 
    $$
    \begin{aligned}
        -\frac{n}{q}\leq \frac{\theta p}{q}\left(1-\frac{n}{\theta p}\right)-\left(1-\frac{\theta p}{q}\right)\frac{n}{\tilde{p}}\quad \iff \quad \frac{n q}{p(n+\tilde{p})}\leq \theta,
    \end{aligned}
    $$
    the assumption in Lemma \ref{lem : Gagliardo-Nirenberg inequality} with $p_1=q$, $p_2=\theta p$, $p_3=\tilde{p}$ and $\vartheta=\frac{\theta p}{q}$ is satisfied for any $\theta \in \left(\frac{ps(s-3)-2(q-p)}{ps(s-3)},1\right)$. Thus, we deduce from Lemma \ref{lem : Gagliardo-Nirenberg inequality} that
    \begin{align*}
    &(2\rho)^\alpha \miint{Q_{2\rho}^\lambda (z_0)} \frac{\Big|u-u_{Q_{2\rho}^\lambda (z_0)}\Big|^{q}}{\left(2\lambda^{\frac{p-2}{2}}\rho\right)^{q}}\, dz\\
    &\qquad\leq c \miint{Q_{2\rho}^\lambda (z_0)} \left(\frac{\Big|u-u_{Q_{2\rho}^\lambda (z_0)}\Big|^{\theta p}}{\left(2\lambda^{\frac{p-2}{2}}\rho\right)^{\theta p} } + |Du|^{\theta p}\right)\, dz\\
    &\qquad\quad \times (2\rho)^\alpha\sup_{I_{2\rho}(t_0)}\left(\dashint_{B_{2\rho}^\lambda (x_0)} \frac{\Big|u-u_{Q_{2\rho}^\lambda (z_0)}\Big|^{\tilde{p}}}{\left(2\lambda^{\frac{p-2}{2}}\rho\right)^{\tilde{p}}}\, dx\right)^{\frac{q-p\theta}{\tilde{p}}}.
    \end{align*}
    The interpolation inequality for $L^p$-norms implies that
    $$
    \begin{aligned}
        &(2\rho)^\alpha\sup_{I_{2\rho}(t_0)}\left(\dashint_{B_{2\rho}^\lambda (x_0)} \frac{\Big|u-u_{Q_{2\rho}^\lambda (z_0)}\Big|^{\tilde{p}}}{\left(2\lambda^{\frac{p-2}{2}}\rho\right)^{\tilde{p}}}\, dx\right)^{\frac{q-p\theta}{\tilde{p}}}\\
        &\qquad \leq \sup_{I_{2\rho}(t_0)} \left(\dashint_{B_{2\rho}(x_0)^\lambda} \frac{\Big|u-u_{Q_{2\rho}^\lambda (z_0)}\Big|^{2}}{\left(2\lambda^{\frac{p-2}{2}}\rho\right)^{2}}\, dx\right)^{\frac{q-p\theta}{2}\tilde{\theta}}\\
        &\qquad\quad \times (2\rho)^\alpha\sup_{I_{2\rho}(t_0)} \left(\dashint_{B_{2\rho}^\lambda(x_0)} \frac{\Big|u-u_{Q_{2\rho}^\lambda (z_0)}\Big|^{s}}{\left(2\lambda^{\frac{p-2}{2}}\rho\right)^{s}}\, dx\right)^{\frac{q-p\theta}{s}(1-\tilde{\theta})},
    \end{aligned}
    $$
    where $\tilde{\theta}=\frac{2(s-\tilde{p})}{\tilde{p}(s-2)}$ and $1-\tilde{\theta}=\frac{s(\tilde{p}-2)}{\tilde{p}(s-2)}$. Note that
    $$
    (q-p\theta)\tilde{\theta}=p(1-\theta)\quad\text{and}\quad (q-p\theta)(1-\tilde{\theta})=q-p.
    $$
    By Lemma \ref{lem : estimate of S in p-intrinsic cylinder of s}, \eqref{eq : decay estimate of rho and lambda} and \eqref{calculate mu_2}, we have
    $$
    \begin{aligned}
        &(2\rho)^\alpha\sup_{I_{2\rho}(t_0)}\left(\dashint_{B_{2\rho}^\lambda (x_0)} \frac{\Big|u-u_{Q_{2\rho}^\lambda (z_0)}\Big|^{\tilde{p}}}{\left(2\lambda^{\frac{p-2}{2}}\rho\right)^{\tilde{p}}}\, dx\right)^{\frac{q-p\theta}{\tilde{p}}}\\
        &\qquad \leq c\lambda^{(q-p\theta)\tilde{\theta}} (2\rho)^\alpha \lambda^{\frac{(2-p)(n+s)}{2s}(q-p\theta)(1-\tilde{\theta})}\rho^{-\frac{n+s}{s}(q-p\theta)(1-\tilde{\theta})}\\
        &\qquad \leq c\lambda^{(1-\theta)p}\rho^{\alpha - \frac{(n+s)(q-p)((n+2)(2-p)+2\mu_2)}{2s\mu_2}}\\
        &\qquad \leq c\lambda^{(1-\theta)p}\rho^{\alpha - \frac{(n+s)(q-p)}{\mu_s}}\\
        &\qquad \leq c\lambda^{(1-\theta)p},
    \end{aligned}
    $$
    where $c=c(\operatorname{data}_s)$.
    Thus, we obtain
    \begin{align*}
    &(2\rho)^\alpha \miint{Q_{2\rho}^\lambda (z_0)} \frac{\Big|u-u_{Q_{2\rho}^\lambda (z_0)}\Big|^{q}}{\left(2\lambda^{\frac{p-2}{2}}\rho\right)^{q}}\, dz \\
    &\qquad\leq c \lambda^{(1-\theta)p}\miint{Q_{2\rho}^\lambda (z_0)} \left(\frac{\Big|u-u_{Q_{2\rho}^\lambda (z_0)}\Big|^{\theta p}}{\left(2\lambda^{\frac{p-2}{2}}\rho\right)^{\theta p} } + |Du|^{\theta p}\right)\, dz.
    \end{align*}

    Therefore, we conclude that for any $\theta \in (\theta_1,1)$,
    $$
    \begin{aligned}
        &\miint{Q_{2\rho}^\lambda (z_0)}\left(\frac{\Big|u-u_{Q_{2\rho}^\lambda (z_0)}\Big|^p}{\left(2\lambda^{\frac{p-2}{2}}\rho\right)^p}+a(z)\frac{\Big|u-u_{Q_{2\rho}^\lambda (z_0)}\Big|^q}{\left(2\lambda^{\frac{p-2}{2}}\rho\right)^q}\right)\, dz\\
        &\qquad \leq c \lambda^{(1-\theta)p} \miint{Q_{2\rho}^\lambda (z_0)}\left(\frac{\Big|u-u_{Q_{2\rho}^\lambda (z_0)}\Big|^{\theta p}}{\left(2\lambda^{\frac{p-2}{2}}\rho\right)^{\theta p}}+|Du|^{\theta p}\right)\, dz\\
        &\qquad\quad + c \lambda^{(1-\theta)p} \miint{Q_{2\rho}^\lambda (z_0)}\inf_{w\in Q_{2\rho}^\lambda (z_0)} a(w)^\theta \left(\frac{\Big|u-u_{Q_{2\rho}^\lambda (z_0)}\Big|^{\theta q}}{\left(2\lambda^{\frac{p-2}{2}}\rho\right)^{\theta q}}+|Du|^{\theta q}\right)\, dz
    \end{aligned}
    $$
    for some $c=c(\operatorname{data}_s)>1$, where 
    $$
    \theta_1=\begin{cases}
        \frac{nq}{p(n+2)}\quad&\text{if }2\leq s\leq 4,\\
        \\
        \max\left\{\frac{n}{n+2},\frac{ps(s-3)-2(q-p)}{ps(s-3)}\right\}\quad &\text{if }4<s<\infty.
    \end{cases}
    $$
\end{proof}
Now, we prove the reverse H\"{o}lder inequality in the $p$-intrinsic case.
\begin{lemma}\label{lem : the reverse Holder inequality in the p-intrinsic case of s}
    There exist constants $c=c(\operatorname{data}_s)>1$ and $\theta_0=\theta_0(n,p,q,s)\in (0,1)$ such that for any $\theta\in (\theta_0,1)$, 
    $$
    \miint{Q_\rho^\lambda (z_0)} H(z,|Du|)\, dz\leq c \left(\miint{Q_{2\rho}^\lambda (z_0)} [H(z,|Du|)]^\theta \, dz\right)^\frac{1}{\theta}+c\miint{Q_{2\rho}^\lambda (z_0)} H(z,|F|)\, dz.
    $$
\end{lemma}
\begin{proof}
    It follows from Lemma \ref{lem : Caccioppoli inequality} that 
    \begin{align}
        \miint{Q_\rho^\lambda (z_0)} H(z,|Du|)\, dz&\leq c \miint{Q_{2\rho}^\lambda (z_0)} \left(\frac{\Big|u-u_{Q_{2\rho}^\lambda (z_0)}\Big|^p}{\left(2\lambda^{\frac{p-2}{2}}\rho\right)^p}+a(z)\frac{\Big|u-u_{Q_{2\rho}^\lambda (z_0)}\Big|^q}{\left(2\lambda^{\frac{p-2}{2}}\rho\right)^q}\right)\, dz \nonumber\\
        &\quad +c\lambda^{p-2} \miint{Q_{2\rho}^\lambda (z_0)} \frac{\Big|u-u_{Q_{2\rho}^\lambda (z_0)}\Big|^2}{\left(2\lambda^{\frac{p-2}{2}}\rho\right)^2}\, dz\nonumber\\ \label{eq : estimation in H in p-intrinsic cylinder of s using Caccioppoli inequality}
        &\quad+c\miint{Q_{2\rho}^\lambda (z_0)} H(z,|F|)\, dz,
    \end{align}
    where $c=c(n,p,q,\nu,L)>1$. Take $\theta_2\coloneq \max\left\{\theta_1,\frac{q-1}{p}\right\}$, where $\theta_1$ is defined in Lemma \ref{lem : estimate of the first term in Lemma 2.1 in p-intrinsic cylinder with s}. For $\theta \in (\theta_2,1)$, using Lemmas \ref{lem : estimate of the first term in Lemma 2.1 in p-intrinsic cylinder with s}, \ref{lem : p-intrinsic parabolic Poincare inequality of p-term in p-intrinsic cylinder with s}, \ref{lem : p-intrinsic parabolic Poincare inequality of q-term in p-intrinsic cylinder with s} and Young's inequality yields
    $$
    \begin{aligned}
        &\miint{Q_{2\rho}^\lambda (z_0)} \left(\frac{\Big|u-u_{Q_{2\rho}^\lambda (z_0)}\Big|^p}{\left(2\lambda^{\frac{p-2}{2}}\rho\right)^p}+a(z)\frac{\Big|u-u_{Q_{2\rho}^\lambda (z_0)}\Big|^q}{\left(2\lambda^{\frac{p-2}{2}}\rho\right)^q}\right)\, dz\\
        &\qquad \leq c \lambda^{(1-\theta)p} \miint{Q_{2\rho}^\lambda (z_0)} [H(z,|Du|)]^\theta \, dz\\
        &\qquad \quad + c \lambda^{\left(1-p+\frac{\alpha \mu_s}{n+s}\right)\theta p+p} \left(\miint{Q_{2\rho}^\lambda (z_0)} (|Du|+|F|)^{\theta p}\, dz\right)^{p-1-\frac{\alpha \mu_s}{n+s}}
    \end{aligned}
    $$
    for some $c=c(\operatorname{data}_s)>1$. Recall that $p-1-\frac{\alpha \mu_s}{n+s}>0$.
    Putting
    $$
    \beta \coloneq \min\left\{p-1-\frac{\alpha\mu_s}{n+s},\frac{1}{2}\right\},
    $$
    we obtain
    \begin{align}
        &\miint{Q_{2\rho}^\lambda (z_0)} \left(\frac{\Big|u-u_{Q_{2\rho}^\lambda (z_0)}\Big|^p}{\left(2\lambda^{\frac{p-2}{2}}\rho\right)^p}+a(z)\frac{\Big|u-u_{Q_{2\rho}^\lambda (z_0)}\Big|^q}{\left(2\lambda^{\frac{p-2}{2}}\rho\right)^q}\right)\, dz \nonumber\\
        &\qquad \leq c \lambda^{(1-\beta\theta)p} \left(\miint{Q_{2\rho}^\lambda (z_0)} [H(z,|Du|)]^\theta \, dz\right)^\beta \nonumber\\\label{eq : estimation of p,q term in H in p-intrinsic cylinder of s using Caccioppoli inequality}
        &\qquad\quad +c \lambda^{(1-\beta\theta)p} \left(\miint{Q_{2\rho}^\lambda (z_0)} H(z,|F|) \, dz\right)^{\beta \theta}.
    \end{align}
    In the same way as in Lemma \ref{lem : the reverse Holder inequality in the p-intrinsic case of bounded solution}, we have
    \begin{align} 
        \lambda^{p-2} \miint{Q_{2\rho}^\lambda (z_0)} \frac{\Big|u-u_{Q_{2\rho}^\lambda (z_0)}\Big|^2}{\left(2\lambda^{\frac{p-2}{2}}\rho\right)^2}\, dz &\leq c\lambda^{p-\beta}\left(\miint{Q_{2\rho}^\lambda (z_0)} [H(z,|Du|)]^\theta\, dz\right)^{\frac{\beta}{\theta p}}\nonumber\\\label{eq : estimation of 2 term in H in p-intrinsic cylinder of s using Caccioppoli inequality}
        &\quad + c\lambda^{p-\beta}\left(\miint{Q_{2\rho}^\lambda (z_0)} H(z,|F|)\, dz\right)^{\frac{\beta}{p}}.
    \end{align}
    Combining \eqref{eq : estimation in H in p-intrinsic cylinder of s using Caccioppoli inequality}, \eqref{eq : estimation of p,q term in H in p-intrinsic cylinder of s using Caccioppoli inequality} and \eqref{eq : estimation of 2 term in H in p-intrinsic cylinder of s using Caccioppoli inequality} implies that for $\theta\in (\theta_0,1)$,
    $$
    \begin{aligned}
        \miint{Q_\rho^\lambda (z_0)} H(z,|Du|)\, dz&\leq c\lambda^{p-\beta}\left(\miint{Q_{2\rho}^\lambda (z_0)} [H(z,|Du|)]^\theta\, dz\right)^{\frac{\beta}{\theta p}}\\
        &\quad +c\lambda^{p-\beta}\left(\miint{Q_{2\rho}^\lambda (z_0)} H(z,|F|)\, dz\right)^{\frac{\beta}{p}},
    \end{aligned}
    $$
    where $\theta_0=\max\{\theta_2,\frac{2n}{(n+2)p}\}$ and $c=c(\operatorname{data}_s)>1$. It follows from Young's inequality that
    $$
    \begin{aligned}
        &\miint{Q_\rho^\lambda (z_0)} H(z,|Du|)\, dz\\
        &\qquad\leq \frac{1}{2}\lambda^p+c\left(\miint{Q_{2\rho}^\lambda (z_0)} [H(z,|Du|)]^\theta\, dz\right)^{\frac{1}{\theta}}+c\miint{Q_{2\rho}^\lambda (z_0)} H(z,|F|)\, dz.
    \end{aligned}
    $$
    Thus, we conclude from \eqref{cond : p-phase condition}$_3$ that
    $$
    \miint{Q_\rho^\lambda (z_0)} H(z,|Du|)\, dz\leq c\left(\miint{Q_{2\rho}^\lambda (z_0)} [H(z,|Du|)]^\theta\, dz\right)^{\frac{1}{\theta}}+c\miint{Q_{2\rho}^\lambda (z_0)} H(z,|F|)\, dz.
    $$
\end{proof}

Lastly, the following lemma will be used in the proof of the gradient higher integrability results. For the proof of this lemma, we refer to \cite[Lemma 3.9]{Wontae2024}.
\begin{lemma}\label{lem : level set in p-intrinsic cylinder}
    Let $u$ be a weak solution to \eqref{eq : the main equation} and assume that $Q_{4\rho}^\lambda(z_0)\subset \Omega_T$ satisfies \eqref{cond : p-phase condition}. Moreover, we assume either \eqref{cond : main assumption with s} or \eqref{cond : main assumption with infty}. Then there exist constants $c=c(\operatorname{data})>1$ and $\theta_0\in (0,1)$ such that for any $\theta\in(\theta_0,1)$,
    $$
    \begin{aligned}
        \iints{Q_{2\kappa \rho}^\lambda (z_0)} H(z,|Du|)\, dz &\leq c\Lambda^{1-\theta} \iints{Q_{2\rho}^\lambda(z_0)\cap \Psi(c^{-1}\Lambda)} [H(z,|Du|)]^\theta\, dz\\
        &\quad + c \iints{Q_{2\rho}^\lambda (z_0)\cap \Phi(c^{-1}\Lambda)} H(z,|F|)\, dz,
    \end{aligned}
    $$
    where $$\theta_0=\begin{cases}
        \theta_0(n,p,q)\quad &\text{if \eqref{cond : main assumption with infty} holds},\\
        \theta_0(n,p,q,s)\quad &\text{if \eqref{cond : main assumption with s} holds}.
    \end{cases}$$
\end{lemma}
\subsection{The $(p,q)$-phase case} 
Let $u$ be a weak solution to \eqref{eq : the main equation} and assume that $G_{2\kappa \rho}^\lambda (z_0)\subset \Omega_T$ satisfies \eqref{cond : p,q-phase condition}. Furthermore, we assume either \eqref{cond : main assumption with infty} or \eqref{cond : main assumption with s}. By \eqref{cond : p,q-phase condition}$_1$, \eqref{cond : p,q-phase condition}$_2$ and \eqref{cond : p,q-phase condition}$_3$, we have
$$
\miint{G_{4\rho}^\lambda  (z_0)} [H_{z_0}(|Du|)+H_{z_0}(|F|)]\, dz < 4 a(z_0)\lambda^q,
$$ 
and hence
$$
\miint{G_{4\rho}^\lambda  (z_0)} [|Du|^q+|F|^q]\, dz < 4 \lambda^q.
$$
The following lemma is a $(p,q)$-intrinsic parabolic Poincar\'{e} inequality, and its proof is similar to that of \cite[Lemma 3.10]{Wontae2024}.
\begin{lemma}\label{lem : (p,q)-parabolic Poincare inequality}
    For $\sigma\in[2\rho,4\rho]$ and $\theta\in \left(\frac{q-1}{p},1\right]$, there exists a constant $c=c(n,p,q,L)>1$ such that
    $$
    \begin{aligned}
        \miint{G_\sigma^\lambda (z_0)} H_{z_0}\left(\frac{\Big|u-u_{G_\sigma^\lambda(z_0)}\Big|}{\lambda^{\frac{p-2}{2} }\sigma}\right)^\theta \, dz&\leq c\Lambda^{(2-p)\theta}\left(\miint{G_\sigma^\lambda (z_0)} [H_{z_0}(|Du|)]^\theta\, dz \right)^{p-1}\\
        &\quad + c\Lambda^{(2-p)\theta}\left(\miint{G_\sigma^\lambda (z_0)} H_{z_0}(|F|)\, dz \right)^{\theta(p-1)}.
    \end{aligned}
    $$
\end{lemma}
Also, as in \cite[Lemma 3.11]{Wontae2024}, by replacing $H_{z_0}(\varkappa)^\theta$ with $\varkappa^{\theta p}$, we obtain the following result.
\begin{lemma}
    For $\sigma \in [2\rho,4\rho]$ and $\theta \in \left(\frac{q-1}{p},1\right]$, there exists a constant $c=c(n,p,q,L)>1$ such that
    $$
    \begin{aligned}
        \miint{G_\sigma^\lambda(z_0)} \left(\frac{\Big|u-u_{G_\sigma^\lambda(z_0)}\Big|}{\lambda^{\frac{p-2}{2} }\sigma}\right)^{\theta p}\, dz&\leq c\lambda^{(2-p)\theta p}\left(\miint{G_\sigma^\lambda(z_0)} |Du|^{\theta p}\, dz\right)^{p-1}\\
        &\quad +c\lambda^{(2-p)\theta p}\left(\miint{G_\sigma^\lambda(z_0)} |F|^{p}\, dz\right)^{\theta(p-1)}.
    \end{aligned}
    $$
\end{lemma}
Next, consider the quadratic term
$$
S(u,G_\rho^\lambda(z_0)) = \sup_{J^\lambda_\rho (t_0)}\dashint_{B_\rho^\lambda (x_0)} \frac{\Big|u-u_{G_\rho^\lambda(z_0)}\Big|^2}{\left(\lambda^{\frac{p-2}{2}}\rho\right)^2}\, dx
$$
in a $(p,q)$-intrinsic cylinder. The proofs of the following lemmas can be found in \cite{Wontae2024}.
\begin{lemma}
    There exists a constant $c=c(n,p,q,\nu,L)>1$ such that
    $$
    S(u,G_{2\rho}^\lambda (z_0)) \leq c \lambda^2.
    $$
\end{lemma}
\begin{lemma}
    There exists a constant $c=c(n,p,q)>1$ such that for any $\theta \in \left(\frac{n}{n+2},1\right)$, 
    $$
    \begin{aligned}
        &\miint{G_{2\rho}^\lambda (z_0)}\left(\frac{\Big|u-u_{G_\rho^\lambda(z_0)}\Big|^p}{\left(2\lambda^{\frac{p-2}{2}}\rho\right)^p}+a(z)\frac{\Big|u-u_{G_\rho^\lambda(z_0)}\Big|^q}{\left(2\lambda^{\frac{p-2}{2}}\rho\right)^q}\right)\, dz\\
        &\qquad \leq c \Lambda^{1-\theta}\miint{G_{2\rho}^\lambda (z_0)}\left(\left[H_{z_0}\left(\frac{\Big|u-u_{G_\rho^\lambda(z_0)}\Big|}{2\lambda^{\frac{p-2}{2}}\rho}\right)\right]^\theta + [H_{z_0}(|Du|)]^\theta\right)\, dz.
    \end{aligned}
    $$
\end{lemma}
\begin{lemma}\label{lem : level set in p,q-intrinsic cylinder}
    There exist constants $c=c(n,p,q,\nu,L)>1$ and $\theta_0=\theta_0(n,p,q)\in (0,1)$ such that for any $\theta\in (\theta_0,1)$,
    $$
    \miint{G_\rho^\lambda(z_0)} H_{z_0}(|Du|)\, dz \leq c\left(\miint{G_{2\rho}^\lambda (z_0)} [H_{z_0}(|Du|)]^\theta\, dz\right)^\frac{1}{\theta}+c\miint{G_{2\rho}^\lambda (z_0)} H_{z_0}(|F|)\,dz.
    $$
    Furthermore, we have
    $$
    \begin{aligned}
        \iints{G_{2\kappa \rho}^\lambda (z_0)} H(z,|Du|)\, dz&\leq c\Lambda^{1-\theta} \iints{G_{2\rho}^\lambda(z_0)\cap \Psi(c^{-1}\Lambda)} [H(z,|Du|)]^\theta \, dz\\
        &\quad+c\iints{G_{2\rho}^\lambda(z_0)\cap \Phi(c^{-1}\Lambda)} H(z,|F|) \, dz.
    \end{aligned}
    $$
\end{lemma}
\section{Proof of the main results}\label{section 5}
In this section, we prove Theorems \ref{thm : main theorem for infty} and \ref{thm : main theorem for s<infty}. First, we construct a Vitali type covering for the collection of intrinsic cylinders defined in Section \ref{section 3}. Thereafter, using this, we complete the proof of Theorems \ref{thm : main theorem for infty} and \ref{thm : main theorem for s<infty}.
\subsection{Vitali type covering argument}\label{subsection 5.1}
For each $w\in \Psi(\Lambda,r_1)$, we consider
$$
\mathcal{Q}(w)\coloneq \begin{cases}
    Q_{2\varrho_w}^{\lambda_w} (w)\quad &\text{if \eqref{case : p-phase} holds,}\\
    G_{2\varsigma_w}^{\lambda_w} (w)\quad &\text{if \eqref{case : p,q-phase} holds,}
\end{cases}
$$
where $\lambda_w$, $\varrho_w$ and $\varsigma_w$ are defined in Section \ref{section 3}. Denote $\ell_w$ as
$$
\ell_w=\begin{cases}
    2\varrho_{w}\quad &\text{if \eqref{case : p-phase} holds,}\\
    2\varsigma_{w}\quad &\text{if \eqref{case : p,q-phase} holds.}
\end{cases}
$$
By following the same argument as in \cite[Subsection 4.2]{Wontae2024}, we obtain a countable collection $\mathcal{G}$ of pairwise disjoint cylinders in $\mathcal{F}\coloneq \{\mathcal{Q}(w): w\in \Psi(\Lambda,r_1)\}$, where $\mathcal{G}$ satisfies the following two conditions:
\begin{itemize}
    \item For each $\mathcal{Q}(z_1)\in \mathcal{F}$, there exists $\mathcal{Q}(z_2)\in \mathcal{G}$ such that
    $$
    \mathcal{Q}(z_1)\cap \mathcal{Q}(z_2)\neq \emptyset.
    $$
    \item For such points $z_1$ and $z_2$, we get
    \begin{equation}\label{eq : comparison of ell_1 and ell_2}
    \ell_{z_1}\leq 2 \ell_{z_2}.
    \end{equation}
\end{itemize}
Then, we only need to prove that for such points $z_1$ and $z_2$,
\begin{equation}\label{eq : covering}
\mathcal{Q}(z_1) \subset \kappa \mathcal{Q}(z_2).
\end{equation}
For this, we want a comparison condition between $\lambda_{z_1}$ and $\lambda_{z_2}$. Indeed, referring to \cite[Subsetion 6.1]{kim2025boundedsolutionsinterpolativegap}, we get
\begin{equation}\label{eq : a comparison condition between lambda_ z_1 and lambda_ z_2}
(4K)^{-\frac{1}{p}}\lambda_{z_1}\leq \lambda_{z_2}\leq (4K)^{\frac{1}{p}}\lambda_{z_1}.
\end{equation}
We show that \eqref{eq : covering} is satisfied in all four possible cases:
\begin{enumerate}[label = (\roman*)]
    \item\label{case : two p-intrinsic cylinders} $\mathcal{Q}(z_2)=Q_{\ell_{z_2}}^{\lambda_{z_2}}(z_2)$\quad and\quad $\mathcal{Q}(z_1)=Q_{\ell_{z_1}}^{\lambda_{z_1}}(z_1)$,
    \item\label{case : two p,q-intrinsic cylinder} $\mathcal{Q}(z_2)=G_{\ell_{z_2}}^{\lambda_{z_2}}(z_2)$\quad and\quad $\mathcal{Q}(z_1)=G_{\ell_{z_1}}^{\lambda_{z_1}}(z_1)$,
    \item\label{case : one p,q-intrinsic cylinder and one p-intrinsic cylinder} $\mathcal{Q}(z_2)=G_{\ell_{z_2}}^{\lambda_{z_2}}(z_2)$\quad and\quad $\mathcal{Q}(z_1)=Q_{\ell_{z_1}}^{\lambda_{z_1}}(z_1)$,
    \item\label{case : one p-intrinsic cylinder and one p,q-intrinsic cylinder} $\mathcal{Q}(z_2)=Q_{\ell_{z_2}}^{\lambda_{z_2}}(z_2)$\quad and\quad $\mathcal{Q}(z_1)=G_{\ell_{z_1}}^{\lambda_{z_1}}(z_1)$.
\end{enumerate}
To prove this, we denote $z_1=(x_1,t_1)$ and $z_2=(x_2,t_2)$ for $x_1,\, x_2\in \Omega$ and $t_1,\,t_2\in (0,T)$. First, we prove the spatial inclusion. Since in all cases, the spatial part of $\mathcal{Q}(z_i)$ $(i=1,\,2)$ is the same as $B_{\ell_{z_i}}^{\lambda_{z_i}}(x_i)$, we only need to show that $B_{\ell_{z_1}}^{\lambda_{z_1}}(x_1)\subset \kappa B_{\ell_{z_2}}^{\lambda_{z_2}}(x_2)$. Indeed, for any $x\in B_{\ell_{z_1}}^{\lambda_{z_1}}(x_1)$, it follows from \eqref{eq : comparison of ell_1 and ell_2} and \eqref{eq : a comparison condition between lambda_ z_1 and lambda_ z_2} that
$$
\begin{aligned}
|x-x_2|&\leq |x-x_1|+|x_1-x_2|\leq 2\ell_{z_1}\lambda_{z_1}^{\frac{p-2}{2}}+\ell_{z_2}\lambda_{z_2}^{\frac{p-2}{2}}\\
&\leq 4(4K)^{\frac{2-p}{2p}}\ell_{z_2}\lambda_{z_2}^{\frac{p-2}{2}}+\ell_{z_2}\lambda_{z_2}^{\frac{p-2}{2}}.
\end{aligned}
$$
Since $\frac{2n}{n+2}< p$ implies $\frac{1}{p}-\frac{1}{2}< \frac{1}{n}<1$, we get
$$
|x-x_2|\leq 4(4K)^{\frac{1}{p}-\frac{1}{2}}\ell_{z_2}\lambda_{z_2}^{\frac{p-2}{2}}+\ell_{z_2}\lambda_{z_2}^{\frac{p-2}{2}}< 17K\ell_{z_2}\lambda_{z_2}^{\frac{p-2}{2}}.
$$
Hence, $B_{\ell_{z_1}}^{\lambda_{z_1}}(x_1)\subset 17K B_{\ell_{z_2}}^{\lambda_{z_2}}(x_2) \subset \kappa B_{\ell_{z_2}}^{\lambda_{z_2}}(x_2)$.

Now, we prove the time inclusion in each case.

\textit{Case $\rm{\ref{case : two p-intrinsic cylinders}}$}. For any $\tau\in I_{\ell_{z_1}}(t_1)$, we have
$$
|\tau-t_2|\leq |\tau - t_1|+|t_1-t_2|\leq 2\ell_{z_1}^2+\ell_{z_2}^2 \leq 9\ell_{z_2}^2<(4\ell_{z_2})^2,
$$
and hence $I_{\ell_{z_1}}(t_1)\subset 4I_{\ell_{z_2}}(t_2)$.

\textit{Case $\rm{\ref{case : two p,q-intrinsic cylinder}}$}. For any $\tau\in J^{\lambda_{z_1}}_{\ell_{z_1}}(t_1)$, we have
$$
|\tau-t_2|\leq |\tau - t_1|+|t_1-t_2|\leq 2\frac{\lambda_{z_1}^p}{\Lambda}\ell_{z_1}^2+\frac{\lambda_{z_2}^p}{\Lambda}\ell_{z_2}^2 \leq (32K+1)\frac{\lambda_{z_2}^p}{\Lambda}\ell_{z_2}^2<\frac{\lambda_{z_2}^p}{\Lambda}(6K\ell_{z_{z_2}})^2,
$$
and hence $J^{\lambda_{z_1}}_{\ell_{z_1}}(t_1)\subset 6K J^{\lambda_{z_2}}_{\ell_{z_2}}(t_2)$.

\textit{Case $\rm{\ref{case : one p,q-intrinsic cylinder and one p-intrinsic cylinder}}$}. In this case, since $\displaystyle K\lambda_{z_1}^p\geq \sup_{Q_{10\varrho_{z_1}}(z_1)} a(\cdot)\lambda_{z_1}^q$, we see from \eqref{eq : a comparison condition between lambda_ z_1 and lambda_ z_2} that
$$
1=\frac{2\lambda_{z_2}^p}{2\lambda_{z_2}^p}\leq \frac{8K\lambda_{z_2}^p}{2\lambda_{z_1}^p} \leq \frac{8K\lambda_{z_2}^p}{\lambda_{z_1}^p + K^{-1}a(z_1)\lambda_{z_1}^q}\leq \frac{8K^2 \lambda_{z_2}^p}{\Lambda}.
$$
Thus, for any $\tau\in I_{\ell_{z_1}}(t_1)$, we have
$$
|\tau-t_2|\leq |\tau - t_1|+|t_1-t_2|\leq 2\ell_{z_1}^2+\frac{\lambda_{z_2}^p}{\Lambda}\ell_{z_2}^2 \leq (64K^2+1)\frac{\lambda_{z_2}^p}{\Lambda}\ell_{z_2}^2<\frac{\lambda_{z_2}^p}{\Lambda}(10K\ell_{z_2})^2,
$$
and hence $I_{\ell_{z_1}}(t_1)\subset 10K J^{\lambda_{z_2}}_{\ell_{z_2}}(t_2)$.

\textit{Case $\rm{\ref{case : one p-intrinsic cylinder and one p,q-intrinsic cylinder}}$}. For any $\tau\in J^{\lambda_{z_1}}_{\ell_{z_1}}(t_1)$, we obtain from \eqref{eq : comparison of ell_1 and ell_2} that
$$
|\tau-t_2|\leq |\tau - t_1|+|t_1-t_2|\leq 2\frac{\lambda_{z_1}^p}{\Lambda}\ell_{z_1}^2+\ell_{z_2}^2 \leq 9\ell_{z_2}^2<(4\ell_{z_2})^2,
$$
and hence $J^{\lambda_{z_1}}_{\ell_{z_1}}(t_1)\subset 4 I_{\ell_{z_2}}(t_2)$. Therefore, we conclude \eqref{eq : covering}.
\subsection{Proof of Theorems \ref{thm : main theorem for infty} and \ref{thm : main theorem for s<infty}}\label{subsection 5.2}
We denote the intrinsic cylinders in the countable pairwise disjoint collection $\mathcal{G}$ by
$$
\mathcal{Q}_k \equiv \mathcal{Q}_k(z_k)\quad (k\in\mathbb{N})
$$
for any $z_k\in \Psi(\Lambda,r_1)$. Using Lemmas \ref{lem : level set in p-intrinsic cylinder} and \ref{lem : level set in p,q-intrinsic cylinder}, we get
$$
    \begin{aligned}
        \iints{\kappa \mathcal{Q}_k} H(z,|Du|)\, dz&\leq c\Lambda^{1-\theta} \iints{\mathcal{Q}_k\cap \Psi(c^{-1}\Lambda)} [H(z,|Du|)]^\theta \, dz\\
        &\quad +c\iints{\mathcal{Q}_k\cap \Phi(c^{-1}\Lambda)} H(z,|F|)\, dz
    \end{aligned}
$$
for any $k\in\mathbb{N}$, where $c=c(\operatorname{data})>1$ and $\theta=\frac{\theta_0+1}{2}$. Here, $$
\theta_0=\begin{cases}
    \theta_0(n,p,q)\qquad & \text{if \eqref{cond : p-phase condition}$_1$ and \eqref{cond : main assumption with infty} hold},\\
    \theta_0(n,p,q,s)\qquad & \text{if \eqref{cond : p-phase condition}$_1$ and \eqref{cond : main assumption with s} hold},\\
    \theta_0(n,p,q)\qquad & \text{if \eqref{cond : p,q-phase condition}$_1$ holds}.\\
\end{cases}
$$  
Using the Vitali type covering argument and Fubini's theorem as in \cite[Subsection 4.3]{Wontae2024}, we deduce that for any $\varepsilon\in (0,\varepsilon_0)$,
$$
\miint{Q_r (z_0)} [H(z,|Du|)]^{1+\varepsilon}\, dz\leq c\Lambda_0^{\varepsilon}\miint{Q_{2r}(z_0)} H(z,|Du|)\, dz+\miint{Q_{2r}(z_0)} [H(z,|F|)]^{1+\varepsilon}\, dz,
$$
where $c=c(\operatorname{data})>1$ and $\varepsilon_0=\varepsilon_0(\operatorname{data})\in (0,1)$. Here, $\Lambda_0$ is defined in \eqref{def : lambda_0 and Lambda_0}. Since $\lambda_0\geq 1$ and $p\leq q$, we have $\Lambda_0^\varepsilon\leq c\lambda_0^{\varepsilon q}$ for some $c=c(\operatorname{data},\|a\|_{L^\infty(\Omega_T)})$. Thus, by the definition of $\lambda_0$, we obtain
$$
\begin{aligned}
    \Lambda_0^{\varepsilon}\miint{Q_{2r}(z_0)} H(z,|Du|)\, dz&\leq c\left(\miint{Q_{2r}(z_0)} H(z,|Du|)\, dz\right)^{1+\frac{2q\varepsilon}{p(n+2)-2n}}\\
    &\quad +c\left(\miint{Q_{2r}(z_0)} [H(z,|F|)+1]^{1+\varepsilon}\, dz\right)^{\frac{2q}{p(n+2)-2n}}.
\end{aligned}
$$
Combining the above inequalities, we complete the proofs of Theorems \ref{thm : main theorem for infty} and \ref{thm : main theorem for s<infty}.\qed

\bibliographystyle{abbrv}
\bibliography{ref}{}
\end{document}